\documentclass{article}
\makeatletter
\newcommand{\@affiliation}{}
\makeatother

\usepackage{arxiv}

\usepackage[utf8]{inputenc} 
\usepackage[T1]{fontenc}    
\usepackage{hyperref}       
\usepackage{url}            
\usepackage{booktabs}       
\usepackage{amsfonts}       
\usepackage{nicefrac}       
\usepackage{microtype}      
\usepackage{lipsum}         
\usepackage{graphicx}
\usepackage{doi}
\usepackage[
]{biblatex}
\newenvironment{proof}{\paragraph{Proof:}}{\hfill$\square$}
\newcommand{\qed}{}
\usepackage{amsmath,amsfonts}
\usepackage{mathtools}
\numberwithin{equation}{section}
\usepackage{subcaption}
\usepackage{cleveref}  

\usepackage{xfrac}
\usepackage{csvsimple}

\renewcommand{\vec}[1]{\boldsymbol{\mathbf{#1}}}
\newcommand{\ds}{\displaystyle}

\newcommand{\bR}{\mathbb{R}}

\renewcommand{\phi}{\varphi}

\DeclareMathOperator{\dd}{d\!}
\DeclareMathOperator{\ddd}{~ d\!}

\newcommand{\del}{\partial}

\newcommand{\dt}{\dd t}

\AtBeginDocument{
\let\div\relax
\DeclareMathOperator{\div}{div}

}

\makeatletter
\newcommand{\oset}[3][0ex]{%
  \mathrel{\mathop{#3}\limits^{
    \vbox to#1{\kern-2\ex@
    \hbox{$\scriptstyle#2$}\vss}}}}
\makeatother

\DeclarePairedDelimiterX{\norm}[1]{\lVert}{\rVert}{
\ifblank{#1}{\:\cdot\:}{#1}
}

\newcommand{\avg}[1]{\ensuremath{\{\mkern-6mu\{#1\}\mkern-6mu\}} }
\newcommand{\jump}[1]{\ensuremath{[\mkern-3mu[#1]\mkern-3mu]} }

\makeatletter
\newcommand{\opnorm}{\@ifstar\@opnorms\@opnorm}
\newcommand{\@opnorms}[1]{%
  \left|\mkern-1.5mu\left|\mkern-1.5mu\left|
   #1
  \right|\mkern-1.5mu\right|\mkern-1.5mu\right|
}
\newcommand{\@opnorm}[2][]{%
  \mathopen{#1|\mkern-1.5mu#1|\mkern-1.5mu#1|}
  #2
  \mathclose{#1|\mkern-1.5mu#1|\mkern-1.5mu#1|}
}
\makeatother

\newcommand{\timed}{\frac{{\text{d}}}{\text{d}t}}
\newcommand{\Solmixed}{{\vec{U}}}

 \newtheorem{thm}{Theorem}[section]
 
 \newtheorem{Corollary}[thm]{Corollary}

 \newtheorem{Definition}[thm]{Definition}
 \newtheorem{Proposition}[thm]{Proposition}
 \newtheorem{rem}[thm]{Remark}

\title{An Energy-Stable Discontinuous Galerkin Method for the Compressible Navier--Stokes--Allen--Cahn System \thanks{This work was supported by the Deutsche Forschungsgemeinschaft (DFG, German Research Foundation) through the project  \mbox{GRK 2160 - Droplet Interaction Technologies} with the project number 270852890 
and the DFG under Germany's Excellence Strategy - EXC 2075 with the project number 390740016.
}}

\newif\ifuniqueAffiliation
\uniqueAffiliationfalse

\ifuniqueAffiliation 
\author{ \href{https://orcid.org/0000-0000-0000-0000}{\includegraphics[scale=0.06]{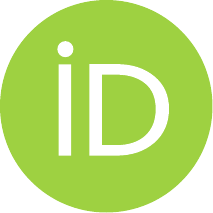}\hspace{1mm}David S.~Hippocampus}\thanks{Use footnote for providing further
		information about author (webpage, alternative
		address)---\emph{not} for acknowledging funding agencies.} \\
	Department of Computer Science\\
	Cranberry-Lemon University\\
	Pittsburgh, PA 15213 \\
	\texttt{hippo@cs.cranberry-lemon.edu} \\
	\And
	\href{https://orcid.org/0000-0000-0000-0000}{\includegraphics[scale=0.06]{orcid.pdf}\hspace{1mm}Elias D.~Striatum} \\
	Department of Electrical Engineering\\
	Mount-Sheikh University\\
	Santa Narimana, Levand \\
	\texttt{stariate@ee.mount-sheikh.edu} \\
}
\else
\usepackage{authblk}

\newbox{\orcid}\sbox{\orcid}{\includegraphics[scale=0.06]{orcid.pdf}} 
\author[]{%
	Lukas Ostrowski\thanks{\texttt{Lukas.Ostrowski@mathematik.uni-stuttgart.de} }  %
}
\author[]{%
	Christian Rohde\thanks{\texttt{Christian.Rohde@mathematik.uni-stuttgart.de}}%
}
\affil[]{Institute of Applied Mathematics and Numerical Simulation, University of Stuttgart}
\fi

\renewcommand{\headeright}{An Energy-Stable  Scheme for the Compressible 
NSAC System}
\renewcommand{\undertitle}{}
\renewcommand{\shorttitle}{}

\hypersetup{
pdftitle={A template for the arxiv style},
pdfsubject={q-bio.NC, q-bio.QM},
pdfauthor={David S.~Hippocampus, Elias D.~Striatum},
pdfkeywords={First keyword, Second keyword, More},
}

\begin{document}
\maketitle

\begin{abstract}
We consider a Navier--Stokes--Allen--Cahn (NSAC) system that governs the compressible motion of a viscous, immiscible two-phase fluid at constant temperature. Weak solutions of the NSAC system dissipate an appropriate energy functional.  Based on an 
equivalent re-formulation of the NSAC system we propose a 
fully-discrete  discontinuous Galerkin (dG) discretization that is mass-conservative,  energy-stable, and provides 
higher-order accuracy  in space and second-order accuracy in time.  The approach relies on the approach in  \cite{Giesselmann2015a} and a special splitting discretization of the derivatives of the free energy function within the Crank-Nicolson time-stepping.\\
Numerical experiments confirm the  analytical statements
and show the applicability of the approach.
\end{abstract}




\section{Introduction}
\label{intro}
For the modelling of the dynamics of a  compressible fluid that occurs in two phase states, say a liquid and a vapour one, there are two different approaches. In the sharp-interface (SI) ansatz the spatial domain is partitioned in two 
bulk domains occupied by  the respective phases. These are coupled by  transmission conditions across the interface which  separates the  bulk domains.
see e.g. \cite{MagieraRohde}.
Alternatively one can rely on diffuse-interface (DI) models. For the compressible regime we refer to  the overview  papers \cite{Anderson1998, Rohde18}. Then,
the interface has a small but  finite thickness. In this interfacial region the different phases might mix to enable a smooth transition from one phase to the other, see Figure \ref{fig:p2:DI-SI} for some illustration.
We are interested in  the subclass of  phase-field models which augment the hydromechanical  equations by an evolution equation for the artificial phase-field variable. The  phase-field variable acts as an indicator for the fluid's phase.
\begin{figure}[ht]
    \centering
    \subcaptionbox{}{
    \includegraphics[width=.4\textwidth]{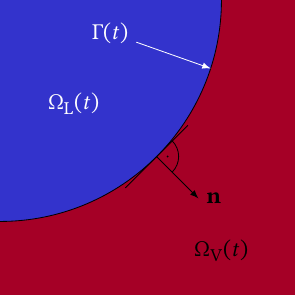}
    }
    \subcaptionbox{}{
   \includegraphics[width=.4\textwidth]{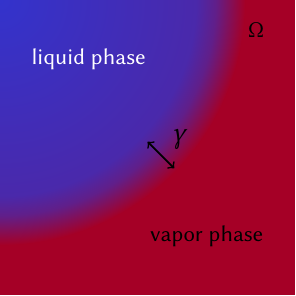}
    }
    \caption{(a) SI approach with time-dependent liquid/vapour bulk domains $\Omega_{L/V}=\Omega_{L/V}(t)$ and separating interface $\Gamma(t)$, (b)   DI approach with ${\mathcal O}(\gamma)$-interfacial width.}
    \label{fig:p2:DI-SI}
\end{figure}
The DI model concepts and thus in particular phase-field  models have 
the advantage that only one system of evolution equations on the entire  domain needs to be solved, without the need for explicit tracking of the interface.  On the other hand, the thermodynamically correct combination of the flow equations with the phase-field dynamics becomes intricate.  Note that 
thermodynamical consistency corresponds to an energy dissipation inequality in the case of isothermal models.\\
A common approach for phase-field modelling is the  coupling of  the  equations of 
hydrodynamics by the stress tensor to a  parabolic phase-field equation of Allen--Cahn type. Such 
Navier--Stokes--Allen--Cahn (NSAC) systems 
have  been introduced in \cite{Blesgen1999} and 
further developped in  \cite{AbelsFeireisl08,Dreyer2014,Ghosh25,HeidaMalek,Witterstein2011}.  Concerning the   well-posedness of weak and strong solutions for these mixed hyperbolic-parabolic systems we refer  to e.g.\ \cite{Feireisl2010,Kotschote2012}. The isothermal NSAC system from 
\cite{Dreyer2014} is the basis for this work, see
\eqref{eq:p2:NSAC} below. 
The authors suggest a free energy density (see \eqref{eq:defenergydensity}) that renders the approach thermodynamically consistent and --notably-- enables mass transfer  between the two phases. In turn, the free energy density 
inhibits a strong nonlinear coupling between density and phase field parameter.

Due  to its mixed-type structure, the numerics for the NSAC models is challenging. We need to resolve steep gradients in the interfacial zone  in a stable and efficient way.  Numerical schemes with artificial dissipation for stabilization can lead to problems like increase of energy or parasitic currents \cite{Coquel2005,Jamet2002}.  Therefore 
we aim to develop a \textit{higher-order and  energy-stable
scheme} to solve  the NSAC system. 
To achieve higher order in space we rely on the discontinuous Galerkin (dG) 
method \cite{BassiRebay,Hesthaven}. More precisely, to treat second-order spatial operators in the original NSAC system we propose a combination of the local 
dG method \cite{Cockburn98}  and  interior-penalty approaches  \cite{di2011mathematical}. Due to their  flexibility, dG methods are frequently used in  the context of compressible DI models
\cite{busto2021high,DKKR16,Kraenkel2018,Repossi2017,tian2016h}. To ensure energy stability in  the discrete level we follow an approach that has been first 
developped in \cite{Giesselmann2014,Giesselmann2015a}
for certain DI models. Our construction starts from a first-order  re-formulation of the NSAC system in terms of variational derivatives of the energy which can then be used as test functions in the dG formulation (see \eqref{eq:p2:mixed} below). The  
design
of tailored stabilizing inter-element fluxes guarentees then mass conservation and energy stability with respect to the spatial discretization.
To retain an at least second-order 
time discretization we utilize a Crank--Nicolson-type time-stepping. The key technical step to ensure energy stability under time discretization is a special  
implicit-explicit  discretization of the stress tensor which involves the specific free energy density discussed above.
Our final main result is then a second-order in time,
higher-order in space  dG method  for the NSAC system that obeys an energy stability 
estimate, see Definition \ref{alg:p2:fully_discrete}. Up to our knowledge, this is the first fully-discrete second-order scheme with this property. \\ 
Notably, our approach is different from the 
construction of higher-order  entropy-dissipative schemes in the realm of e.g.\ hyperbolic systems that are equipped with a convex entropy  \cite{lefloch2000highorder,tadmor03}.
The mixed-type structure of the NSAC system prevents the use of this elegant approach. But 
there are various other related works on the numerics for Navier-Stokes--Allen--Cahn  systems.   
In \cite{HE2021}, a dG method is proposed  generalizing the scalar auxiliar variable  approach which is widely used in the incompressible  realm.   The achieved energy stability is restricted to a first-order time discretization. In turn,
higher-order in time is achieved for spatially continuous approaches in \cite{Cao23}. The works  \cite{FeireislPetcuShe23} and \cite{FeireislPetcuShe21} refer to thermodynamically consistent discretizations based on  finite-volume ideas. The first dG discretization for the NSAC system has been proposed in  \cite{Kraenkel2018} but is also restricted to first-order accuracy in time. 

This work is structured as follows. In Section \ref{chap:p2:NSAC} we review the isothermal NSAC system from \cite{Dreyer2014}, that allows for phase transition dynamics.  Further, we introduce a mixed formulation with corresponding notion of weak solutions, and show that these dissipate the corresponding energy (Theorem \ref{thm:p2:energy_ineq}).  
In the subsequent Section \ref{chap:p2:DG} we propose the energy-stable dG method discussing first  semi-discretizations in space and in time, respectively (Theorem \ref{thm:p2:spatialdisc_conservation} and Theorem \ref{thm:p2:energy_ineq_discrete}).The main rigorous result is Theorem \ref{thm:p2:discrete_energy_ineq} which establishes 
mass conservation and energy stability for a higher-order fully-discrete dG method. 
In Section \ref{sec:NumExp} we present numerical experiments which demonstrate the higher-order accuracy and  confirm the theoretical findings on energy stability. In several examples we 
illustrate the capability of the model and our dG method. 
Preliminary numerical results of this work can be found in the conference paper \cite{Massa20}. It relies on the PhD 
thesis of the first author \cite{Ostrowski}.

\section{A Phase-Field Approach to Compressible Two-Phase Flow}\label{chap:p2:NSAC}
\subsection{The Compressible Navier--Stokes--Allen--Cahn System}
We consider  the dynamics of a viscous fluid at constant temperature for some time interval $[0,T]$, $T>0$, that 
occupies a domain $\Omega \subset \mathbb{R}^d, \ d\in \mathbb{N}$. The hydromechanical unknowns are the density of the fluid denoted by  $\rho=\rho(\vec{x},t) >0 $  and the velocity field $\vec{v} = \vec{v}(\vec{x},t)\in \mathbb{R}^d$.\\
The fluid is assumed to exist in two phases, a liquid phase denoted by subscript $\mathrm{L}$ and a vapor phase denoted by subscript $\mathrm{V}$.
In each phase the fluid is thermodynamically completely described by the corresponding  Helmholtz free energy density $\rho f_{\mathrm{L/V}}(\rho)$, which is assumed to be given. Since we target at a phase-field description of the two-phase dynamics we introduce as further unknown the phase-field variable        $\varphi = \varphi    (\vec{x},t) \in [0,1]$,  which is supposed to take the value $0$ ($1$)
for a pure liquid (vapour) state.\\ 
Now,  following \cite{Dreyer2014} we will introduce as governing system  the isothermal compressible Navier\---Stokes\---Allen\---Cahn (NSAC) system. With the phase-field parameter $\gamma > 0$, which relates to the interfacial width, the Helmholtz free energy density for the two-phase fluid is supposed to be given by
\begin{align}
    \rho f(\rho,\varphi,\nabla\varphi) &= h(\varphi)\rho f_\mathrm{L}(\rho)+(1-h(\varphi)) \rho f_\mathrm{V}(\rho) + \frac{1}{\gamma}W(\varphi) + \frac{\gamma}{2}|\nabla\varphi|^2. \label{eq:p2:rhof}
\end{align}
It consists of the interpolated free energy densities $\rho f_{\mathrm{L/V}}$ of the pure liquid and vapor phases. These are assumed to be convex functions with bounded derivatives.  
For the  nonlinear interpolation function $h=h(\varphi)$ and the  Van-der-Waals  energy \cite{Cahn1958} using the double well 
potential,
we make the polynomial choices 
\begin{align}\label{eq:p2:def_h}
    h(\varphi) = 3 \varphi^2 - 2 \varphi^3,\quad
    W(\varphi)= a \varphi^2(1-\varphi)^2 \quad (a>0).
\end{align}
For the numerical analysis in Section \ref{chap:p2:DG}  the specific choices in \eqref{eq:p2:def_h} are not essential but the use of (smooth) polynomials renders the notion of weak solutions below to be well-posed, see \eqref{eq:regu}. 
Let us define the mixture energy densities
\begin{align}
    \rho \psi(\varphi,\rho) &\coloneqq h(\varphi)\rho f_\mathrm{L}(\rho)+(1-h(\varphi)) \rho f_\mathrm{V}(\rho) \quad \text{ and } \\
    \rho \tilde{f}(\varphi, \rho) &\coloneqq \rho \psi(\varphi,\rho) + \frac{1}{\gamma} W(\varphi).
    \label{eq:defhomoenergy}
\end{align}
Then we can rewrite \eqref{eq:p2:rhof} in the compact form
\begin{align}\label{eq:defenergydensity}
      \rho f(\rho,\varphi,\nabla\varphi) &= \rho \psi(\varphi,\rho) + \frac{1}{\gamma}W(\varphi) + \frac{\gamma}{2} |\nabla \varphi|^2 
      = \rho \tilde{f}(\varphi,\rho) + \frac{\gamma}{2} |\nabla \varphi|^2.
\end{align}
The hydrodynamic pressure $p$ is defined 
through the Helmholtz free energy $\rho f$ using 
the Gibbs-like relation
\begin{equation}\label{eq:p2:pdef}
    p=p(\rho,\varphi,\nabla\varphi) = -\rho f(\rho,\varphi,\nabla \varphi)+\rho \frac{\partial (\rho f)}{\partial \rho}(\rho, \varphi,\nabla \varphi).
\end{equation}
Furthermore,  we introduce  the generalized chemical potential by
\begin{equation}
    \mu = \frac{1}{\gamma}W'(\varphi)+ \frac{\partial (\rho\psi)}{\partial \varphi}-\gamma \Delta \varphi,
\end{equation}
which steers the phase-field variable into equilibrium.
Additionally, we denote by $\eta>0$ the (artificial) mobility.\\
The isothermal compressible NSAC system in $\Omega \times (0,T)$ reads then as
\begin{align}
    \partial_t \rho + \div(\rho \vec{v}) &= 0,\nonumber\\ 
    \partial_t(\rho \vec{v}) + \div(\rho \vec{v}\otimes \vec{v}+{p\vec{I}})&= \div(\vec{S}) - \gamma \div(\nabla \varphi\otimes \nabla \varphi), \label{eq:p2:NSAC}\\
    \rho \partial_t \varphi + \rho \nabla \varphi \cdot \vec{v} &= -\eta \mu. \nonumber
\end{align}
Here, we denote the  identity matrix in 
${\mathbb{R}}^{d\times d}$ by $\vec{I} $. The viscous part of the stress tensor is given by
\begin{align}\label{def:stressinterpol}
    \vec{S}=\vec{S}(\varphi,\nabla\vec{v})= \nu(\varphi) (\nabla \vec{v}+ \nabla \vec{v}^\top - \div(\vec{v})\vec{I}),
\end{align}
with an interpolation of the given  viscosities $\nu_{\mathrm{L/V}}>0$ of the pure phases:
\begin{align}
    \nu(\varphi) = h(\varphi)\nu_\mathrm{L} + (1-h(\varphi))\nu_\mathrm{V} > 0.
    \label{eq:p2:viscointerpol}
\end{align}
Note that we have set for the sake of simplicity in \eqref{def:stressinterpol} the bulk and the  shear viscosity to be equal to 1. \\
Assuming an impermeable no-slip wall and $90^\circ$ static contact angle, we impose  the boundary conditions
\begin{align}
        \vec{v}(\cdot,t) &= \vec{0}  \text{ and } 
        \nabla \varphi(\cdot,t) \cdot \vec{n} = 0   { \text{ on } \partial \Omega, t\in[0,T].} \label{eq:p2:bc2}
\end{align}
By $\vec{n} \in {\mathcal S}^{d-1}$ we denote the outer normal of $\partial \Omega$.
Additionally, the system is endowed with initial conditions
\begin{align}
\label{eq:p2:IC}
    \rho(\cdot,0) &= \rho_0, \quad \vec{v}(\cdot,0) = \vec{v}_0, \quad \varphi(\cdot,0) = \varphi_0 \quad \text{ in } \Omega,
\end{align}
using suitable functions $(\rho_0,\vec{v}_0, \varphi_0)\colon \Omega \to \mathbb{R}^+ \times \mathbb{R}^d \times [0,1]$.\\
The total energy of the system \eqref{eq:p2:NSAC} at time $t\in [0,T]$ is defined as the sum of  bulk free energy and kinetic energy, i.e.,
\begin{equation}\label{eq:p2:etot}
    E(t) \coloneqq 
    %
 \int_\Omega \rho(\vec{x},t) f\big(\rho(\vec{x},t),\varphi(\vec{x},t),\nabla \varphi(\vec{x},t)\big) + \frac{1}{2} \rho(\vec{x},t) |\vec{v}(\vec{x},t)|^2 \ddd\vec{x}.
\end{equation}
Introducing the  bilinear form 
\[
B[\varphi;\cdot, \cdot]  : 
(\vec{v},\vec{X})    
\mapsto 
\ds \int_\Omega    {\bf S} (\varphi;\nabla \vec{v}) : \nabla \vec{X}   \dd\vec{x}
\]
for some  $\varphi$ with $0\le \varphi \le 1$,
it can be readily shown that smooth solutions of 
\eqref{eq:p2:NSAC} that satisfy \eqref{eq:p2:bc2}    in $(0,T) $ fulfill  the energy dissipation law
\begin{equation}\label{energydecay}
\frac{\dd}{{\rm d}t}E(t) =-\int_\Omega \frac{\eta}{\rho} \mu^2 \dd\vec{x} - B[\varphi;\vec{v},\vec{v}] \leq 0.
\end{equation}

\begin{rem} 
    \begin{enumerate}
     \item The inequality \eqref{energydecay} renders the 
     the NSAC system \eqref{eq:p2:NSAC} thermodynamically consistent. We will prove  the analogous energy dissipation 
     rate for a mixed formulation of \eqref{eq:p2:NSAC} 
    in Theorem \ref{thm:p2:energy_ineq} below.   
        \item
        The polynomial function $h$  from \eqref{eq:p2:def_h} with the interpolation property  $h(0)=0$, $h(1)=1$ satisfies in particular $h'(0) = h'(1) \allowbreak \neq 0$. This guarantees that \eqref{eq:p2:NSAC} allows for physical meaningful equilibria like a static single-phase equilibrium $\rho =  const., \vec{v} = \boldsymbol{0}, \varphi \equiv 0$. If $h'(0)\neq 0$ holds,  then the right hand side of the phase\-field equation $\eqref{eq:p2:NSAC}_3$  would not vanish.
        In \cite{Dreyer2014} it has been shown that the choices \eqref{eq:p2:def_h} are compatible with the sharp-interface limit $\gamma\to 0$ in the sense that the limit obeys well-known transmission conditions for phase transition. 
%
%
 \item 
 Up to our knowledge there are no
  results on the well-posedness of   solutions for \eqref{eq:p2:NSAC}. But 
    there are results on the existence of solutions to similar systems which mainly differ in the scaling of the second-order operators. The scalings are needed in order to ensure the boundedness of the 
    phase-field variable.
    In \cite{Feireisl2010} the existence of global-in-time weak solutions is proven, and  in \cite{Kotschote2012} the authors prove the existence and uniqueness of local strong solutions.
    \end{enumerate}
\end{rem}

%


\subsection{A Mixed Formulation of the NSAC System and Energy Stability}\label{sec:p2:mixedenergy}
Our  dG method    will rely on a mixed, non-conservative formulation  of the NSAC system \eqref{eq:p2:NSAC} 
which will be proposed in this section.
Notably, the mixed system is of first order if the viscous 
stress tensor is neglected.
Further we introduce a notion of weak solutions, and it is 
shown  in Theorem \ref{thm:p2:energy_ineq}   that  weak solutions  dissipate an energy functional that corresponds to $E$ from \eqref{eq:p2:etot}. 

First, using  the free energy density $\rho f$ from \eqref{eq:p2:rhof} with the homogeneous  part  $\rho \tilde{f}$ from \eqref{eq:defhomoenergy},
we introduce auxiliary variables
\begin{align}\label{eq:auxilary}
    \vec{\sigma} &= \nabla \varphi, \quad 
    \mu = \frac{\partial (\rho \tilde{f})}{\partial \varphi} - \div(\gamma\vec{\sigma}), \quad 
    \tau  = \frac{\partial (\rho \tilde{f})}{\partial \rho} + \frac{1}{2} |\vec{v}|^2.
\end{align}
Note that $\mu$ and $\tau$ are the (variational) derivatives of  the  integrand  of the  total energy in \eqref{eq:p2:etot}   with respect to $\varphi$ and $\rho$. 
With these, we rewrite the NSAC system \eqref{eq:p2:NSAC} into a mixed formulation and search for 
\begin{equation}\label{defU}
\vec{U} \coloneqq (\rho,\vec{v}, \varphi, \mu,\tau,\vec{\sigma} ),
\end{equation}
that satisfies 
\renewcommand{\ds}{\displaystyle}
\begin{equation}\label{eq:p2:mixed}
 \begin{array}{rcl} 
  \ds  \partial_t \rho + \div(\rho\vec{v}) &= &0, \\
  \ds   \rho \partial_t \vec{v} + \div(\rho \vec{v}\otimes \vec{v}) -\div(\rho \vec{v})\vec{v} - \frac{1}{2}\rho\nabla|\vec{v}|^2 - \mu \nabla \varphi+\rho \nabla \tau
    &= &\div(\vec{S}),  \\
 \ds    \partial_t \varphi + \nabla \varphi\cdot \vec{v}  &=& \ds - \eta \frac{\mu}{\rho},\\
   \ds     \mu -\frac{\partial (\rho \tilde{f})}{\partial \varphi}+\gamma \div(\vec{\sigma})&=& 0,  \\
 \ds   \tau - \frac{\partial (\rho \tilde{f})}{\partial \rho} - \ds \frac{1}{2}|\vec{v}|^2 &=&0, \\
  \vec{\sigma} - \nabla \varphi&=&0.
\end{array}
\end{equation}
%
In terms of the mixed-system variables the total energy $E$ from \eqref{eq:p2:etot} is expressed as 
\begin{equation}\label{hatenergy}
\hat E(t) :=\int_\Omega \hspace*{-0.015cm}\rho(\vec{x},t) \tilde f\big(\rho(\vec{x},t),\varphi(\vec{x},t)\big) +   \hspace*{-0.015cm} \frac{\gamma}{2} |\vec{\sigma}(\vec{x},t)|^2 +   \hspace*{-0.015cm}   \frac{1}{2} \rho(\vec{x},t) |\vec{v}(\vec{x},t)|^2 \dd \vec{x}.
\end{equation}
The initial conditions  \eqref{eq:p2:IC} and the boundary conditions
\begin{align}
        \vec{v}(\cdot,t) &= \vec{0}  \text{ and } 
       \vec{\sigma}(\cdot,t) \cdot \vec{n} = 0   { \text{ on } \partial \Omega, t\in[0,T]} \label{eq:p2:bc3}
\end{align}
complete the system.
Note that  \eqref{eq:p2:mixed} involves as primary unknown the velocity as primitive variable. The non-conservative re-formulation of the momentum equations is needed to prove the discrete energy stability statements in Section 
\ref{chap:p2:DG}.\\
Let $L^2(\Omega)$ and  $H^k(\Omega)$, $k\in \mathbb{N}$, denote  the standard Lebesgue and  Sobolev spaces, respectively.  To cope with the boundary conditions \eqref{eq:p2:bc3} in a weak function space setting  we define  
\begin{equation}
\label{eq:defspace2}
\begin{array}{rcl}
    H_0^1(\Omega) &\coloneqq& \{ \phi \in H^1(\Omega) \colon \phi\vert_{\del \Omega} = 0 \},\\[1.1ex]
    H_{\vec{n}}^1(\Omega) &\coloneqq& \{ \vec{\phi} \in \left(H^1(\Omega)\right)^d \colon \vec{\phi}\vert_{\del \Omega} \cdot \vec{n} = 0 \},
\end{array}
\end{equation}
and the spatial solution space as
\[
{\mathcal V}=     H^1(\Omega) \times   (H_0^1(\Omega))^d \times H^1(\Omega)  
\times H^1(\Omega)  \times H^1(\Omega) \times     H_{\vec{n}}^1(\Omega).
\]
 We  proceed to define a weak solution. In the sequel we use abbreviations  like e.g.,~$\rho(t)$  instead of  $\rho(\cdot,t)$  for the density. 
\begin{Definition}[Weak solution of the mixed formulation \eqref{eq:p2:mixed}]\label{def:weaksol}
    \ \\
  The  function  $ \Solmixed:= (\rho,\vec{v},\varphi,\mu,\tau,\vec{\sigma})\in C^0([0,T];{\mathcal V} )$  with
  \begin{equation}\label{eq:regu}
  \begin{array}{c}\rho(t) > 0,\,   \varphi(t) \in [0,1] \text{ a.e.},\, 
  \rho^{-1}(t), \,  (\rho \tilde f)(t), \,   \rho(t) {|\vec{v}|(t)}^2  \in L^2(\Omega),  \\[1.1ex]
  \del_t \rho(t),    \del_t \varphi(t)   \in L^2(\Omega),         \del_t\vec{v}(t),{\del_t\vec{\sigma}}(t) \in  (L^2(\Omega))^d
  \end{array}
  \end{equation}
for all $t \in [0,T]$  is called a \textit{weak solution of the initial boundary value problem for  \eqref{eq:p2:mixed}}
if it satisfies the initial conditions in \eqref{eq:p2:IC} and if  we have  in  $ (0,T)$ the relations
    \begin{equation}
    \begin{array}{rcl}
        0 &=& \ds \int_\Omega (\del_t \rho + \div(\rho\vec{v})) \psi \dd \vec{x},
       \\[1.7ex]
        0 &=& \ds \int_\Omega \Big(\vphantom{\frac{1}{2}} \rho \del_t \vec{v} + \div(\rho\vec{v}\otimes \vec{v}) - \div(\rho\vec{v})\vec{v}   \\[1.5ex]
        &&\ds \phantom{\int_\Omega} + \rho \nabla \tau - \mu\nabla \varphi -\frac{1}{2}\rho \nabla|\vec{v}|^2\Big) \cdot \vec{X} \dd \vec{x} + B[\varphi;\vec{v}(t), \vec{X}] ,\\[1.7ex]
        0 &=& \ds \int_\Omega \left( \del_t \varphi + \vec{v}\cdot \nabla \varphi + \eta \frac{\mu}{\rho} \right) \Theta \dd \vec{x}, \\[1.7ex]
        0 &=& \ds \int_\Omega \Big( \mu-\frac{\del \rho \tilde{f}}{\del \varphi}(\rho,\varphi)+\gamma \div(\vec{\sigma}) \Big) \chi \dd \vec{x}, \\[1.7ex]
        0 &=& \ds \int_\Omega \Big( \tau - \frac{\del \rho \tilde{f}}{\del \rho}(\rho,\varphi)-\frac{1}{2}|\vec{v}|^2 \Big) \zeta \dd \vec{x},\\[1.7ex]
        0 &=& \ds \int_\Omega (\vec{\sigma}-\nabla \varphi) \cdot \vec{Z} \dd \vec{x}
    \end{array}   \label{eq:p2:weak}
    \end{equation}for all test functions  $(\psi,\vec{X},\Theta,\chi,\zeta,\vec{Z}) \in \mathcal{V}$.
\end{Definition}
The conditions in \eqref{eq:regu} ensure not only the existence of all integrals in \eqref{eq:p2:weak} but 
also in the subsequent energy notions. Recall that weak solutions of the original system \eqref{eq:p2:NSAC} satisfy a  total energy inequality, see \cite{Giesselmann2014}. We show that weak solutions of \eqref{eq:p2:mixed} dissipate the energy $\hat E$ from \eqref{hatenergy} with the corresponding   entropy production.

\begin{thm}[Energy stability of the weak solution]\label{thm:p2:energy_ineq}\ \\
Let $\vec{U}= (\rho,\vec{v},\varphi, \mu,\tau,\vec{\sigma})$  be a weak solution to the initial boundary value problem for \eqref{eq:p2:mixed}.  Then for all $t \in (0,T)$, the following energy inequality holds:
    \begin{equation}
    \begin{aligned} \label{eq:p2:energy_ineq}
        \frac{\dd}{\dd t} \hat E(t) &
         =\frac{\dd}{\dd t} \left(\int_\Omega \rho \tilde f(\rho,\varphi) +    \frac{\gamma}{2} |\vec{\sigma}|^2 +    \frac{1}{2} \rho |\vec{v}|^2 \ddd\vec{x}\right)  \\
        &=-\int_\Omega \frac{\eta}{\rho} \mu^2 \dd\vec{x} - B[\varphi;\vec{v},\vec{v}] \leq 0.
    \end{aligned}
    \end{equation}
\end{thm}
As expected, the entropy production   is driven by the chemical potential and the viscous forces. To motivate  the 
energy analysis for the numerical method in Section \ref{chap:p2:DG}  we provide the proof of the theorem.
\begin{proof}(of Theorem \ref{thm:p2:energy_ineq}) 
    In a straightforward way we compute:
    \begin{align*}
        \frac{\dd}{\dd t} \hat E(t) =& \frac{\dd}{\dd t} \left(\int_\Omega \rho \tilde f(\rho,\varphi) +  \frac{\gamma}{2} |\vec{\sigma}|^2 +   \frac{1}{2}\rho |\vec{v}|^2 \dd\vec{x} \right) \\
        =& \int_\Omega  \del_t \rho\frac{\partial (\rho \tilde f)}{\partial \rho}  +
       \del_t \varphi \frac{\partial (\rho \tilde f)}{\partial \varphi}  +\gamma\del_t \vec{\sigma\!} \cdot \vec{\sigma}
            + \frac{1}{2}\del_t \rho  |\vec{v}|^2     + (\rho\vec{v})\cdot \del_t \vec{v} \dd\vec{x}\\
     =& \int_\Omega  \del_t \rho \Big(\frac{\partial (\rho \tilde f)}{\partial \rho}
        - \frac{1}{2} |\vec{v}|^2\Big)   
             +\del_t \varphi \Big(
             \frac{\partial (\rho \tilde f)}{\partial \varphi} - \gamma \div \vec{\sigma}   \Big)\\[1.5ex] 
    &{}  \hspace*{13em} + \del_t \rho  |\vec{v}|^2     + (\rho\vec{v})\cdot \del_t \vec{v} \dd\vec{x}.
        %
        %
\end{align*}        
For the last line we used the 
 time derivative of $\eqref{eq:p2:weak}_6$
with $\vec{Z} = \vec\sigma \in H^1_{\vec{n}} (\Omega) $.  The split of the kinetic energy term allows to use 
$\eqref{eq:p2:weak}_{4,5}$  with  $\chi = \del_t \varphi, \zeta = \del_t \rho$ resulting in 
\begin{align*}
       { \frac{\dd}{\dd t} \hat E(t)} &=
   \int_\Omega  \del_t \rho \tau  +       \del_t \varphi \mu - \div(\rho \vec{v}) |\vec{v}|^2     +  \vec{v}\cdot (\rho\del_t \vec{v} )
           \dd\vec{x}\\
          &=     \int_\Omega  \del_t \rho\tau  +       \del_t \varphi \mu - \div(\rho \vec{v}) |\vec{v}|^2  \\ 
& \hspace*{1.9em}- \vec{v} \cdot  ( \div(\rho\vec{v}\otimes \vec{v})) + \div(\rho \vec{v})|\vec{v}|^2  \\
&\hspace*{1.9em}-\rho\nabla \tau \cdot \vec{v}
+ \mu \nabla \varphi \cdot \vec{v}   + \frac12\rho\nabla |\vec{v}|^2 \cdot\vec{v}  \dd\vec{x}\\[1.2ex] 
&\phantom{=}{} -B[\varphi;\vec{v}, \vec{v}].      
\end{align*}       
Here we have used the weak relation for the velocity $\vec{v}$ in $\eqref{eq:p2:weak}$. In the next step 
we plug in the time derivatives of $\rho$ and $\varphi$ and use the choice of the ansatz spaces to deduce 
\begin{align*}
    \ds \frac{\dd}{\dd t} \hat E(t)
          &=     \int_\Omega  -\div (\rho \vec{v}) \tau  -  (\vec{v}\cdot \nabla \varphi)\mu  -  \eta \frac{\mu^2}{\rho} - \div(\rho \vec{v}) |\vec{v}|^2  \\ 
& \hspace*{1.9em}- \vec{v} \cdot  (\div( \rho\vec{v}\otimes \vec{v})) + \div(\rho \vec{v})|\vec{v}|^2 \\
& \hspace*{1.9em} -\rho\nabla \tau \cdot \vec{v}
+ \mu \nabla \varphi \cdot \vec{v}   + \frac12\rho\nabla |\vec{v}|^2 \cdot\vec{v}  \dd\vec{x}\\
&\phantom{=}{} -B[\varphi;\vec{v}, \vec{v}] \\
     &= \int_\Omega    -  \eta \frac{\mu^2}{\rho} \dd \vec{x}    -B[\varphi;\vec{v}, \vec{v}]. 
\end{align*}  
For the last line we used  $\vec{v}\in H_{\vec{n}}^1(\Omega)$.
\qed
\end{proof}


\section{An Energy-Stable dG Method for the NSAC  System}\label{chap:p2:DG}

The derivation of our scheme to solve the system \eqref{eq:p2:NSAC} is based on \cite{Giesselmann2014}. In contrast to \cite{Kraenkel2018}, where an energy-consistent discontinuous Galerkin (dG) method for a similar model has been derived, we obtain a numerical scheme with second order accuracy in time, instead of order one. The resulting dG method is by construction mass conservative.

Based on the mixed formulation from Section \ref{sec:p2:mixedenergy}, we suggest in Section  \ref{sec:p2:spatial} a  spatial semi-discretization and in Section \ref{sec:p2:temporal} the temporal semi-discretization. These two are combined in Section \ref{sec:p2:fully_discrete} to obtain the final fully-discrete scheme,  given in Definition \ref{alg:p2:fully_discrete}.

\subsection{A Semi-Discrete dG Method for the  NSAC System}\label{sec:p2:spatial}

For the  spatial semi-discretization of \eqref{eq:p2:mixed} we   introduce mesh notations.
Let $\mathcal{T}$ be a conforming triangulation of $\Omega$. That means $\mathcal{T}=\{T\}$ is a finite set of $d$-simplices  such that
\begin{enumerate}
    \item $T \in \mathcal{T}$ implies $T$ is an open set,
    \item for any $T_1,T_2 \in \mathcal{T}$ we have $\bar{T}_1 \cap \bar{T}_2$ is a  sub-simplex of both, $\bar{T}_1$ and $\bar{T}_2$,
    \item $\bigcup_{T\in\mathcal{T}} \bar{T} = \bar{\Omega}$.
\end{enumerate}
The {mesh size} $h>0$ of $\mathcal{T}$ is defined as
\begin{align}\label{eq:p2:meshsizedef}
    h\coloneqq \max_{T \in \mathcal{T}} h_T,
\end{align}
where $h_T$ denotes the diameter of the  simplex $T$.
    For some mesh $\mathcal{T}$, we call a subset $e$ of $\bar{\Omega}$ a \emph{mesh face}, if $e$ has positive $(d-1)$-dimensional Hausdorff measure and if exactly one of the following conditions is fulfilled:
    \begin{enumerate}
        \item  There are  two distinct mesh elements $T_1,T_2 \in \mathcal{T}$ such that $e = T_1\cap T_2$. In this case we call $e$ an \emph{interface}.
        \item
       There exists $T\in \mathcal{T}$ with $e=\del T\cap \del \Omega.$ Then we call $e$ a \emph{boundary face}.
    \end{enumerate}
Let $\mathcal{E}$ denote  the set of interfaces of the triangulation $\mathcal{T}$.\\ 
We define further for $k \in \mathbb{N}$ the broken Sobolev space
    \begin{align}\label{eq:p2:broken_sobolev_def}
        H^k(\mathcal{T})\coloneqq \left\{ \phi \in L^2(\Omega) \colon \forall \, T \in \mathcal{T}, \phi\vert_T \in H^k(T) \right\},
    \end{align}
    and similarly the sets $H^1_0(\mathcal{T})$ and $H_{\vec{n}}^1(\mathcal{T})$ based on \eqref{eq:defspace2}.\\
    In order to use functions defined in broken spaces restricted to the interface set $\mathcal{E}$ of the triangulation, we define the trace space
    \begin{align}\label{eq:p2:trace_space_def}
        \operatorname{Tr}(\mathcal{E})\coloneqq \prod_{T\in \mathcal{T}} L^2(\del T).
    \end{align}
Equipped with these  notations and based on the notion of weak solutions from Definition \ref{def:weaksol} we are now able to present a semi-discrete method.\\ 
For the  ansatz and test spaces of the  discontinuous Galerkin method let 
\newcommand{\Vzero}{{V}_{h,0} }
\newcommand{\Vnormal}{{V}_{h,\vec{n}} }
\begin{equation}
\begin{array}{c}
    V_h \coloneqq \left\{ u \in L^2(\Omega) \colon  u\vert_T\in \mathbb{P}^k\,\,   \forall \, T \in \mathcal{T} \right    \}, \\[1.1ex]
    \Vzero  \coloneqq V_h \cap H_0^1(\mathcal{T}), \quad 
    \Vnormal  \coloneqq (V_h)^d \cap H_{\vec{n}}^1(\mathcal{T}),
\end{array}
\end{equation}
where $\mathbb{P}^k$ is the space of polynomials up to degree $k$.
Further, let 
\begin{equation}\label{functionspace}
\mathcal{V}_h = V_h \times (\Vzero)^d\times V_h \times V_h \times V_h \times \Vnormal.
\end{equation}
For the discretization of the initial data from \eqref{eq:p2:IC} we introduce 
the mapping
$
\Pi_h :  {\mathcal V} \to {\mathcal V}_h
$
as the $L^2$-projection to the ansatz space ${\mathcal V}_h$. Now, given  initial functions  $\rho_0 \in H^1(\Omega)$, 
$\vec{v}_0 \in {(H^1_0(\Omega))}^d$,  $\varphi_0\in H^1(\Omega)$ with ${\sigma}_0 := \nabla \varphi_0 \in H^1_{\vec{n}}(\Omega)$  in \eqref{eq:p2:IC} and 
\begin{equation}
\label{iniremainder}
\mu_0:=  \frac{\partial (\rho \tilde{f})}{\partial \varphi}(\varphi_0,\rho_0) -\gamma \div(\vec{\sigma}_0),\quad 
\tau_0 :=   \frac{\partial (\rho \tilde{f})}{\partial \rho}(\varphi_0,\rho_0) +  \frac{1}{2}|\vec{v}_0|^2,
\end{equation}
we define the initial data in ${\mathcal V}_h$ by
\begin{equation}\label{inidatadiscrete}
\vec{U}_{0,h} =  (\rho_{0,h}, \vec{v}_{0,h}, \varphi_{0,h}, \mu_{0,h}, \tau_{0,h}, \vec{\sigma}_{0,h}) := \Pi_h (\rho_{0}, \vec{v}_{0}, \varphi_{0}, \mu_{0}, \tau_{0,h}, \vec{\sigma}_{0}).
\end{equation}
Note that $\vec{U}_{0,h}$ is well-defined due to the regularity of $\rho\tilde f$.\\
We proceed with  jump and average operators for some scalar-valued function 
 $\Phi \in V_h$ 
and some vector-valued function $\vec{u} \in {(V_h)}^d$. 
    Let $T_1,T_2$ be two mesh  elements  in $\mathcal T$ with a common facet and let $T\in \mathcal T$ be a 
      simplex with some boundary face. Then we define 
    %
    \begin{equation}\label{def:p2:jump_avg}
    \begin{array}{rcl}
        \avg{\Phi}_{e}   &\coloneqq& 
         \left\{ \begin{array}{c@{\quad :\quad}l}  
                 \hspace*{1.85em}(\Phi\vert_{T_1} + \Phi\vert_{T_2})/2   \hspace*{1.9em}&  e\in  \mathcal{E} \cap \bar T_1 \cap  \bar T_2,\\ 
              \Phi\vert_T & e\in \partial \Omega \cap \bar T,
        \end{array}
        \right.
        \\[2.1ex]
        \avg{\vec{u}}_{e} &\coloneqq&
         \left\{ \begin{array}{c@{\quad :\quad}l}  
                 \hspace*{1.9em}  (\vec{u}\vert_{T_1} + \vec{u}\vert_{T_2})/2  \hspace*{1.9em} &  e\in  \mathcal{E} \cap \bar T_1 \cap \bar T_2,\\ 
              \vec{u}\vert_T & e\in \partial \Omega \cap \bar T,
        \end{array}
        \right.
        \\[2.1ex]
        \jump{\Phi}_{e} &\coloneqq& 
         \left\{ \begin{array}{c@{\quad :\quad}l}  
                \hspace*{1.2em} \Phi\vert_{T_1} \vec{n}_{T_1} + \Phi\vert_{T_2} \vec{n}_{T_2}   \hspace*{1.2em} &  e\in  \mathcal{E} \cap \bar T_1 \cap \bar T_2,\\ 
              \jump{\Phi}_{e_b} & e\in \partial \Omega \cap \bar T,
        \end{array}
        \right.        
        \\[2.1ex]
        \jump{\vec{u}}_{e} &\coloneqq &
        \left\{ \begin{array}{c@{\quad :\quad}l} 
         \hspace*{0.5em} \vec{u}\vert_{T_1} \cdot \vec{n}_{T_1} + \vec{u}\vert_{T_2} \cdot \vec{n}_{T_2}   \hspace*{0.5em}&  e\in  \mathcal{E} \cap \bar T_1 \cap \bar T_2,\\ 
         \vec{u}\vert_{T} \cdot \vec{n}_{T} & e\in \partial \Omega \cap \bar T,
        \end{array}
        \right.       
        \\[2.1ex]
        \jump{\vec{u}}_{\otimes,e} &\coloneqq&
         \left\{ \begin{array}{c@{\quad :\quad}l}  
          \vec{u}\vert_{T_1} \otimes \vec{n}_{T_1} + \vec{u}\vert_{T_2} \otimes \vec{n}_{T_2}
         &  e\in  \mathcal{E} \cap \bar T_1 \cap \bar T_2,\\ 
      \vec{u}\vert_{T} \otimes \vec{n}_{T} & e\in \partial \Omega \cap \bar T.
        \end{array}
        \right.      
    \end{array}
    \end{equation}
    We omit the subscript $e$ whenever no confusion can arise and use for some subset 
    $ {\mathcal E}' \subseteq \mathcal E \cap \partial \Omega$    the notation 
    \[
          \sum_{e \in \mathcal{E}'} \int_e    \avg{\Phi}_{e}(s)     \dd s  = \int_{{\mathcal E}'} \avg{\Phi}_{e}(s) \dd s,
    \]
    analogously for  $ \avg{\vec{u}}_{e},\,     \jump{\Phi}_{e},\,     \jump{\vec{u}}_{e}, \, \jump{\vec{u}}_{\otimes,e}$.
%
\newcommand{\Solsemi}{\vec{U}_h}  
\newcommand{\Dummysemi}{\vec{W}_h}   
For the discretization of the  viscous stress tensor $\vec{S} $  in \eqref{eq:p2:mixed}, we use  for some 
$\varphi_h \in V_h$ with $0\le \varphi \le 1 $ a.e. 
the  discrete bilinear form $B_h[\varphi_h;\cdot,\cdot] \colon   {(\Vzero)}^d  \times  {(\Vzero)}^d\to \bR$ given by 
  \begin{equation}\label{eq:bilinearform}
        \begin{array}{rcl}
            \lefteqn{B_h[\varphi_h;\vec{v}_h,{\vec{X}_h}] :=  B[\phi_h;\vec{v}_h,{\vec{X}_h}]} \\[1.5ex]
            &&\ds {}-  
            \int_{\mathcal{E}\cup \partial \Omega} \avg{\vec{S}(\nabla\vec{v}_h,\varphi_h)} : \jump{\vec{\vec{X}_h}}_\otimes
            + \avg{\vec{S}(\nabla \vec{X}_h,\varphi_h)} : \jump{\vec{v}_h}_\otimes \dd s \\[2.5ex]
           & &\ds {}+
           \int_{\mathcal{E}\cup \partial \Omega} \frac{\alpha_B}{|e|}\jump{\vec{v}_h}_\otimes : \jump{\vec{\vec{X}_h}}_\otimes \dd s   \qquad \qquad  \qquad \big(\vec{v}_h, \vec{X}_h \in {(\Vzero)}^d\big).
        \end{array}
    \end{equation}
Here  $\alpha_B>0$  is a parameter which is assumed  to be big enough to ensure the coercivity  of $B_h$, see Remark \ref{rem:coercive} 
below.\\ 
Based on these notations we  introduce an abstract version  of a spatially semi-discrete dG method  which leaves the choice of discrete fluxes still open. 
\begin{Definition}[General spatially semi-discrete dG method]\label{alg:p2:spacesemi} \ \\ 
Let some mappings
\begin{equation}\label{def:F}
\begin{array}{rclcc}
     F_1,F_3,F_4,F_5 &\colon& \mathcal{V}_h \times V_h &\to& L^2(\mathcal{E}), \\
     F_2 &\colon& \mathcal{V}_h \times (\Vzero)^d &\to&L^2(\mathcal{E}), \\
     F_6 &\colon& \mathcal{V}_h \times    \Vnormal &\to&L^2(\mathcal{E})
 \end{array}
 \end{equation}
be given that satisfy the consistency condition 
\begin{equation} \label{flus}
F_i[\vec{W}, \cdot ] =0 \quad \big(i\in  \{1,\ldots,6\}\big)
\end{equation}
for all  continuous functions $\vec{W} \in  {\mathcal V} \subset {\mathcal V}_h$. \\  
The  function  $ \Solsemi:= (\rho_h,\vec{v}_h,\varphi_h,\mu_h,\tau_h,\vec{\sigma}_h) \in C^0( [0,T];{\mathcal V}_h) $  with
$ \Solsemi(\cdot,0) =   \vec{U}_{0,h} $ and 
  \begin{equation}\label{eq:reguh}
  \begin{array}{c}\rho_h(t) > 0,\,   \varphi_h(t) \in [0,1]\text{ a.e.},\\[1.1ex]
  \rho^{-1}_h(t), \,  (\rho_h \tilde f(\varphi_h,\rho_h))(t), \,   \rho_h(t) {|\vec{v}_h(t)|}^2  \in L^2(\Omega),  \\[1.1ex]
  \del_t \rho_h(t),    \del_t \varphi_h(t)   \in L^2(\Omega),         \del_t \vec{v}_h(t),{ \del_t\vec{\sigma_h}}(t) \in  (L^2(\Omega))^d
  \end{array}
  \end{equation}
for all $t \in [0,T]$  is called a \textit{general semi-discrete dG approximation}
if it satisfies 
    \begin{equation}
    \begin{array}{rcl}
        0 &=& \ds \int_\Omega (\del_t \rho_h + \div(\rho_h\vec{v}_h)) \psi_h \ddd \vec{x} + \int_{\mathcal{E}} F_1[\Solsemi(t),\psi_h] \dd s,
       \\[1.9ex]
        0 &=& \ds \int_\Omega \Big(\vphantom{\frac{1}{2}} \rho_h \del_t \vec{v}_h + \div(\rho_h\vec{v}_h\otimes \vec{v}_h) - \div(\rho_h\vec{v}_h)\vec{v}_h   \\[1.7ex]
        &&\ds \phantom{\int_\Omega} + \rho_h \nabla \tau_h - \mu_h\nabla \varphi_h -\frac{1}{2}\rho_h \nabla|\vec{v}_h|^2\Big) \cdot \vec{X}_h \dd \vec{x} \\[1.7ex]
        &&\ds {}+ \int_{\mathcal{E}} F_2[\Solsemi(t),\vec{X}_h] \ddd s + B_h[\varphi;\vec{v}_h, \vec{X}_h], \\[1.9ex]
        0 &=& \ds \int_\Omega \Big( \del_t \varphi_h + \vec{v}_h\cdot \nabla \varphi_h + \eta \frac{\mu_h}{\rho_h} \Big) \Theta_h \dd \vec{x} \\[1.7ex]
        &&\ds {} + \int_{\mathcal{E}} F_3[\Solsemi(t),\Theta_h] \dd s, \\[1.7ex]
        0 &=& \ds \int_\Omega \Big( \mu_h-\frac{\del \rho \tilde{f}}{\del \varphi}(\rho_h,\varphi_h)+\gamma \div(\vec{\sigma}_h) \Big) \chi_h \dd \vec{x} \\[1.7ex]
        &&\ds {}+ \int_{\mathcal{E}} F_4[\Solsemi(t),\chi_h] \dd s, \\[1.7ex]
        0 &=& \ds \int_\Omega \Big( \tau_h - \frac{\del \rho \tilde{f}}{\del \rho}(\rho_h,\varphi_h)-\frac{1}{2}|\vec{v}_h|^2 \Big) \zeta_h \dd \vec{x} \\[1.7ex]
        &&\ds  {}+ \int_{\mathcal{E}} F_5[\Solsemi(t),\zeta_h] \dd s, \\[1.7ex]
        0 &=& \ds \int_\Omega (\vec{\sigma}_h-\nabla \varphi_h) \cdot \vec{Z}_h \dd \vec{x}
        + \int_{\mathcal{E}} F_6[\Solsemi(t),\vec{Z}_h] \dd s
    \end{array}   \label{eq:p2:spatialsemi}
    \end{equation}
    for all  $t\in (0,t)$ and all  test functions  $(\psi_h,\vec{X}_h,\Theta_h,\chi_h,\zeta_h,\vec{Z}_h) \in \mathcal{V}_h$.
\end{Definition}

\begin{rem}\label{rem:coercive}
    \begin{enumerate}
        \item The spaces for the  semi-discrete  formulation are chosen in a way that all occurring integrals in \eqref{eq:p2:spatialsemi} are well-defined. In the sequel, we assume the existence of a semi-discrete dG approximation $\vec{U}_h$. As for  \eqref{eq:p2:weak} there is no well-posedness analysis for this system. 
        \item 
        For the sake of maintaining the symmetry of the exact bilinear form
         for the viscous stress tensore $\vec{S}$,
        and to obtain a discrete coercivity property, we rely for the  discrete bilinear  form $B_h$ on the  symmetric interior penalty discretization \cite{Arnold2002}. Coercivity is acchieved for choosing $\alpha >0$ big enough in \eqref{eq:bilinearform}.
    \end{enumerate}
\end{rem}

To prove the mass conservation and energy consistency of the semi-discrete dG approximation satisfying \eqref{eq:p2:spatialsemi}, we recall the following result   on discrete partial integration using the space 
 \begin{align*}
        H^{\div}(\mathcal{T}) \coloneqq \Big\{ \vec{u}_h \in (L^2(\mathcal{T}))^d \colon \div(\vec{u}_h) \in L^2(\mathcal{T}) \Big\}.
    \end{align*}
It follows directly from elementwise partial integration and the definition of the jump and average operators in \eqref{def:p2:jump_avg}.

\begin{Proposition}[Elementwise integration]\label{prop:p2:elem_int}
    If $\vec{u}_h\in H^{\div}(\mathcal{T})$ and $\varphi_h \in H^1(\mathcal{T})$, then
    \begin{align*}
        \sum_{T \in \mathcal{T}} \int_T \div(\vec{u}_h)\varphi_h \dd \vec{x} &= \sum_{T \in \mathcal{T}} \left( - \int_T \vec{u_h}\cdot \nabla \varphi_h \dd \vec{x} + \int_{\del T} \varphi_h \vec{u}_h\cdot \vec{n}_T \dd s \right) \\
        &= \sum_{T \in \mathcal{T}} - \int_T \vec{u}_h\cdot \nabla \varphi_h \dd \vec{x}  + \int_{\mathcal{E}\cup \del \Omega} \jump{\vec{u}_h}\avg{\varphi_h} \dd s \\
        &\phantom{=\,}        + \int_{\mathcal{E}\cup \del \Omega} \jump{\varphi_h}\avg{\vec{u}_h} \dd s \\
        &= - \int_{\Omega} \vec{u}_h\cdot \nabla \varphi_h \ddd \vec{x}+ \int_{\mathcal{E}\cup \del \Omega} \jump{\varphi_h\vec{u}_h} \dd s.
    \end{align*}
\end{Proposition}


 Next, we present necessary and sufficient conditions on the elementwise fluxes from 
\eqref{def:F} to ensure mass conservation and energy stability
for the general spatially  semi-discrete dG approximation.

\begin{thm}[Semi-discrete energy]\label{thm:p2:spatialdisc_conservation}
 For a  general spatially semi-discrete dG approximation  $\Solsemi:= (\rho_h,\vec{v}_h,\varphi_h,\mu_h,\tau_h,\vec{\sigma}_h) \in C^0( [0,T];{\mathcal V}_h)$ 
according to Definition \ref{alg:p2:spacesemi}, the following statements hold. 
    \begin{enumerate}
        \item  For $\vec{U}_h$ holds mass conservation in (0,T), i.e.\ 
    \begin{equation}\label{semidGmassconserve}
      \frac{\dd}{\dd t}  \int_\Omega \rho_h\dd \vec{x} =0
    \end{equation}
     if and only if we have 
        \begin{align*}
             \int_{\mathcal{E}} F_1[\vec{U}_h,1] \dd s &= - \int_{\mathcal{E}} \jump{\rho_h \vec{v}_h} \dd s.
        \end{align*}

        %
        %
        \item  For $\vec{U}_h$ holds the energy dissipation equation
        \begin{align}
        \begin{split}
         \frac{\dd}{\dd t}   \hat E_h(t) &:=\frac{\dd}{\dd t} \left( \int_\Omega \rho_h \tilde{f}(\rho_h,\varphi_h) +\frac{\gamma}{2}|\vec{\sigma}_h|^2 + \frac{\rho_h}{2} |\vec{v}_h|^2 \dd \vec{x} \right) \\[1.7ex]
            &= -\int_{\Omega} \frac{\eta}{\rho_h}\mu_h^2 \dd \vec{x} -  B_h[\varphi_h;\vec{v}_h,\vec{v}_h]  \dd \vec{x} 
            \leq 0 
            \label{eq:p2:energy_dissipation_discrete}
        \end{split}    
        \end{align}
        if and only if the conditions 
        \begin{equation} \label{eq:p2:energy_dissipation_discrete_cond1}
        \begin{aligned}
            0 =& - \int_{\mathcal{E}} \gamma   \frac{\dd}{\dd t} F_6[\vec{U}_h,\vec{\sigma}_h] \dd s+\int_{\mathcal{E}}\gamma \jump{\del_t \varphi_h \vec{\sigma}_h} \dd s \\
           & 
         - \int_{\mathcal{E}} F_4[\vec{U}_h,\del_t \varphi_h] \dd s,
           \\
            0 =&  \int_{\mathcal{E}} F_1[\vec{U}_h,\tau_h]\dd s
            + \int_{\mathcal{E}} F_2[\vec{U}_h,\vec{v}_h] \dd s  \\
            &+ \int_{\mathcal{E}} F_3[\vec{U}_h,\mu_h]\dd s
            +  \int_{\mathcal{E}} \jump{\rho_h\vec{v}_h\tau_h} \dd s,
        \end{aligned}
         \end{equation}
         and $F_5 \equiv 0 $ hold in  $(0,T)$. 
    \end{enumerate}
\end{thm}

\begin{proof}
 \phantom{ddd}\\
    1. Using  $\psi_h \equiv 1 \in V_h $ in the first relation in \eqref{eq:p2:spatialsemi} we  obtain  $  \frac{\dd}{\dd t} \int_\Omega \rho_h \dd \vec{x} =0$ if and only if  
        \begin{align*}
            0 =& - \int_\Omega \div(\rho_h \vec{v}_h) \ddd \vec{x} 
            - \int_{\mathcal{E}} F_1[\vec{U}_h,1] \dd s. \\
            \intertext{We infer with Proposition \ref{prop:p2:elem_int} the relations}
            0=& -\int_{\mathcal{E}\cup \del \Omega} \jump{\rho_h\vec{v}_h} \avg{1} \ddd s - \int_{\mathcal{E}} F_1[\vec{U}_h,1] \dd s.
        \end{align*}
        With $\vec{v}_h \equiv\vec{0}$ on $\del \Omega$ this is the first assertion of the theorem.\\
        2. To show \eqref{eq:p2:energy_dissipation_discrete} for the discrete energy $  \hat E_h(t)$
          we explicitly compute the time derivative
        \begin{align*}
            \timed    \hat E_h(t) =& \int_\Omega  \Big(\del_t \rho_h\frac{\del (\rho\tilde{f})}{\del \rho}(\rho_h,\varphi_h) 
            + \del_t \varphi_h \frac{\del (\rho\tilde{f})}{\del \varphi}(\rho_h,\varphi_h) + \gamma \del_t \vec{\sigma}_h \cdot \vec{\sigma}_h \\
            &\phantom{\int_\Omega}  \hspace*{9.5em} +\frac{1}{2}\del_t \rho_h |\vec{v}_h|^2+ \rho_h\vec{v}_h \cdot \del_t \vec{v}_h  \Big)\dd \vec{x}.\\
        \end{align*}
            Taking the time derivative of the last relation in \eqref{eq:p2:spatialsemi}  and then choosing  $\vec{Z}_h= \vec{\sigma}_h(t)$ for arbitrary but fixed $t$ we obtain
        \begin{align*}
            \lefteqn{\hspace*{-2em}\timed   \hat E_h (t) }\\
              =&\int_\Omega \Big(\del_t \rho_h\frac{\del (\rho\tilde{f})}{\del \rho}(\rho_h,\varphi_h) 
            + \del_t \varphi_h\frac{\del (\rho\tilde{f})}{\del \varphi}(\rho_h,\varphi_h)  + \gamma \vec{\sigma}_h \cdot \nabla (\del_t \varphi_h)\\
            &\phantom{\int_\Omega} \hspace*{14em}+\frac{1}{2}\del_t \rho_h |\vec{v}_h|^2 + \rho_h\vec{v}_h \cdot \del_t \vec{v}_h   \Big)\dd \vec{x} \\
            & - \int_{\mathcal{E}}\gamma   \frac{\dd}{\dd t} F_6[\vec{U}_h,\vec{\sigma}_h] \dd s.
            \end{align*}
            Applying Proposition \ref{prop:p2:elem_int}  and $\vec{\sigma}_h\in  V_{h,\vec{n}} $ yields
            \begin{align*}
            \lefteqn{\hspace*{-2em}\timed   \hat E_h(t))}\\ 
           =& \int_\Omega \Big(  \del_t \rho_h\frac{\del (\rho\tilde{f})}{\del \rho}(\rho_h,\varphi_h) + \del_t \varphi_h\frac{\del (\rho\tilde{f})}{\del \varphi}(\rho_h,\varphi_h)- \gamma \del_t \varphi_h \div(\vec{\sigma}_h)
           \\
            &\hspace*{16em}  +  \frac{1}{2}\del_t \rho_h |\vec{v}_h|^2  + \rho_h\vec{v}_h \cdot \del_t \vec{v}_h\Big) \dd \vec{x} \\
            &- \int_{\mathcal{E}}  \gamma   \frac{\dd}{\dd t} F_6[\vec{U}_h(t),\vec{\sigma}_h] \dd s
            + \int_{\mathcal{E} }  \gamma \jump{(\del_t \varphi_h)\vec{\sigma}_h} \dd s.
            %
            \end{align*}
            {Like in the proof of Theorem \ref{thm:p2:discrete_energy_ineq} using the  dG equations for $\mu_h$ and $\tau_h$ in \eqref{eq:p2:spatialsemi} we obtain}
            \begin{align*} \timed   \hat E_h(t)) 
            =& \int_\Omega \Big( \del_t \rho_h\tau_h  + \del_t \varphi_h\mu_h 
            %
            -\vec{v}_h\cdot\div(\rho_h\vec{v}_h\otimes \vec{v}_h) + \div(\rho_h\vec{v}_h)|\vec{v}_h|^2\\ 
           & \hspace*{1em}- \rho_h  \nabla \tau_h \cdot\vec{v}_h+  \mu_h \nabla \varphi_h\cdot \vec{v}_h+ \frac{1}{2} \rho_h \nabla |\vec{v}_h|^2 \cdot\vec{v}_h  \Big) \dd \vec{x} \\
            &+ \int_{\mathcal{E}} F_4[\vec{U}_h,\del_t \varphi_h] \dd s + \int_{\mathcal{E}} F_5[\vec{U}_h,\del_t \varrho_h] \dd s \\
            &-  \int_{\mathcal{E}} \gamma   \frac{\dd}{\dd t} F_6[\vec{U}_h,\vec{\sigma}_h]
            +\int_{\mathcal{E}} \gamma \jump{(\del_t \varphi_h)\vec{\sigma}_h} \dd s \\
            &
            - \int_{\mathcal{E}} F_2[\vec{U}_h,\vec{v}_h] \dd s-  B_h[\varphi_h;\vec{v}_h,\vec{v}_h].  \\
            \intertext{We eliminate the time derivatives of $\rho_h$ and $\varphi_h$ in the volume integral
            with the respective evolutions equations in  \eqref{eq:p2:spatialsemi}. This results in }
            \timed   \hat E_h(t))=& \int_\Omega \Big(-\tau_h \div(\rho_h\vec{v}_h) - \mu_h \nabla \varphi_h \cdot \vec{v}_h - \eta\frac{|\mu_h|^2}{\rho_h} \\
            &\phantom{\int_\Omega \Big( \, }- \div(\rho_h \vec{v}_h \otimes \vec{v}_h) \vec{v}_h + \div(\rho_h\vec{v}_h) |\vec{v}_h|^2 - \rho_h \vec{v}_h \cdot \nabla \tau_h \\
            &\hspace*{9.5em}+ \mu_h \vec{v}_h \cdot \nabla \varphi_h + \frac{1}{2} \rho_h \vec{v}_h \cdot \nabla |\vec{v}_h|^2\Big) \dd \vec{x} \\
            &+ \int_{\mathcal{E}} F_4[\vec{U}_h,\del_t \varphi_h] \dd s  + \int_{\mathcal{E}} F_5[\vec{U}_h,\del_t \rho_h] \dd s\\
            &- \int_{\mathcal{E}} \gamma   \frac{\dd}{\dd t} F_6[\vec{U}_h,\vec{\sigma}_h]
            + \int_{\mathcal{E}} \gamma\jump{(\del_t \varphi_h)\vec{\sigma}_h} \dd s \\
            %
            %
            &- \int_{\mathcal{E}} F_2[\vec{U}_h,\vec{v}_h] \dd s
            - B_h[\varphi_h;\vec{v}_h,\vec{v}_h) \\
            &- \int_{\mathcal{E}} F_3[\vec{U}_h,\mu_h] \dd s
            - \int_{\mathcal{E}} F_1[\vec{U}_h,\tau_h] \dd s. \\
            \end{align*}
            {Finally, using integration by parts via Proposition \ref{prop:p2:elem_int} and $\vec{v}_h\in  (V_{h,0})^d$,  we obtain}
            \begin{equation}\label{eq:in:proof}
            \begin{array}{rcl}\ds
            \timed   \hat E_h((t))&=& \ds \int_{\mathcal{E}} F_4[\vec{U}_h,\del_t \varphi_h] \dd s
                                             +\int_{\mathcal{E}} F_5[\vec{U}_h,\del_t \rho_h] \dd s \\[2.0ex]
            &&\ds {}- \int_{\mathcal{E}}\gamma    \frac{\dd}{\dd t} F_6[\vec{U}_h,\vec{\sigma}_h] \dd s
            +  \int_{\mathcal{E}}\gamma  \jump{\del_t \varphi_h \vec{\sigma}_h} \dd s  \\[2.0ex]
            %
            %
            &&\ds-\int_{\mathcal{E}\cup } \jump{\rho_h\vec{v}_h\tau_h} \dd s - \int_{\mathcal{E}} F_1[\vec{U}_h,\tau_h] \dd s \\[2.0ex]
            &&\ds- \int_{\mathcal{E}} F_2[\vec{U}_h,\vec{v}_h] \ddd s
            - \int_{\mathcal{E}} F_3[\vec{U}_h,\mu_h] \dd s  \\[2.0ex]
            &&\ds- B_h[\varphi_h;\vec{v}_h,\vec{v}_h]- \int_{\Omega} \eta \frac{\mu_h^2}{\rho_h}  \dd \vec{x}. \\[2.0ex]
            \end{array}
            \end{equation}
            {To fulfill the discrete energy dissipation equation \eqref{eq:p2:energy_dissipation_discrete}, we need}
            \begin{align*}
            0 =&  \int_{\mathcal{E}} F_4[\vec{U}_h,\del_t \varphi_h] \dd s  +\int_{\mathcal{E}} F_5[\vec{U}_h,\del_t \rho_h] \dd s\\
            &- \int_{\mathcal{E}} \gamma   \frac{\dd}{\dd t} F_6[\vec{U}_h,\vec{\sigma}_h] \dd s
                + \int_{\mathcal{E}}\jump{\del_t \varphi_h \vec{\sigma}_h} \dd s  \\
            %
            %
            &-\int_{\mathcal{E}} \jump{\rho_h\vec{v}_h\tau_h} \dd s - \int_{\mathcal{E}} F_1[\vec{U}_h,\tau_h] \dd s \\
            &- \int_{\mathcal{E}} F_2[\vec{U}_h,\vec{v}_h] \dd s
            - \int_{\mathcal{E}} F_3[\vec{U}_h,\mu_h) \dd s.
        \end{align*}
         Since the traces of $\del_t \varphi_h$  as well as  $\del_t \rho_h$  are independent of the traces of the other quantities, the terms containing these terms  and the ones without them cannot cancel each other.
        This yields    $F_5\equiv 0$  and the conditions in \eqref{eq:p2:energy_dissipation_discrete_cond1}.\qed
    
\end{proof}

The  generic  (but non-unique) choice of consistent fluxes given by 
\begin{equation}\label{eq:fluxchoice}
\begin{array}{rcl}
    F_1[\vec{U}_h,\psi_h] &=& - \jump{\rho_h\vec{v}_h} \avg{\psi_h}, \\[1.0ex]
    F_2[\vec{U}_h,\vec{X}_h] &=& - \jump{\tau_h} \avg{\rho_h \vec{X}} + \jump{\varphi_h} \avg{\mu_h \vec{X}}, \\[1.0ex]
    F_3[\vec{U}_h,{\Theta}_h] &=& - \jump{\varphi_h}\avg{\Theta \vec{v}_h}, \\[1.0ex]
    F_4 [\vec{U}_h,{\chi}_h]&=& - \gamma \jump{\vec{\sigma}_h} \avg{\chi}, \\[1.0ex]
    F_5 [\vec{U}_h,{\zeta}_h]&=& 0, \\[1.0ex]
    F_6 [\vec{U}_h,{\vec{Z}}_h]&=& \jump{\varphi_h} \avg{\vec{Z}}
\end{array}
\end{equation}
for $\vec{U}_h, \,  (\psi_h,\vec{X}_h,\Theta_h,\chi_h,\zeta_h,\vec{Z}_h) \in \mathcal{V}_h$ 
fulfills the conditions in \eqref{eq:p2:energy_dissipation_discrete_cond1} of  Theorem \ref{thm:p2:discrete_energy_ineq}. The corresponding semi-discrete dG approximation dissipates then
exactly the discrete counterpart of the energy dissipated by a weak solution of the  NSAC system, see \eqref{eq:p2:energy_dissipation_discrete}  and \eqref{eq:p2:energy_ineq}. 
For stabilization reasons we can add appropriate dissipative fluxes to the ones in \eqref{eq:fluxchoice} such that the energy  $  \hat E_h$ is  still dissipated.  
Restricting stabilizing terms to the evolution equations for the density, the velocity, and  the phase-field variable we get the following  final result of the
section. 

\begin{Corollary}[Stabilized  semi-discrete dG approximation]\label{cor:stable}
Let  para\-meters $\alpha_1,\alpha_2,\alpha_3 \ge 0$ be given. We consider the fluxes $F_4,F_5,F_6$ as in \eqref{eq:fluxchoice} but choose $F_1,F_2,F_3$ according to 
\begin{align}
    F_1 &= - \jump{\rho_h\vec{v}_h} \avg{\psi}+ \alpha_1 \jump{\tau_h}\jump{\psi},  \label{eq:p2:flux1}\\
    F_2 &= - \jump{\tau_h} \avg{\rho_h \vec{X}} + \jump{\varphi_h} \avg{\mu_h \vec{X}} + \alpha_2 \jump{\vec{v}_h}\jump{\vec{X}}\label{eq:p2:flux2},\\
    F_3 &= - \jump{\varphi_h}\avg{\Theta \vec{v}_h} + \alpha_3 \jump{\mu_h}\jump{\Theta} \label{eq:p2:flux3}.
\end{align}
Then, a general semi-discrete dG  approximation $\vec{U}_h$ satisfies the energy dissipation inequality 
\begin{align}
        \begin{split}
         \frac{\dd}{\dd t}   \hat E_h(t)
            &= -\int_{\Omega} \frac{\eta}{\rho_h}\mu_h^2 \dd \vec{x} -  B_h[\varphi_h;\vec{v}_h,\vec{v}_h]  \dd \vec{x} \\
           &\phantom{=\,\,} - \int_{\mathcal{E}}\alpha_1 {\jump{\tau_h}}^2+   \alpha_2 {\jump{\vec{v}_h}}^2  +  \alpha_3 {\jump{\mu_h}}^2 \dd s
            \leq 0.            \label{eq:p2:energy_dissipation_discrete_stable}
        \end{split}    
        \end{align}
\end{Corollary}
\begin{proof}
We follow the proof of Theorem \ref{thm:p2:spatialdisc_conservation} and observe 
that  the  second argument in the  fluxes $F_1,F_2,F_3$ 
in \eqref{eq:in:proof}  are taken by the quantities $\tau_h$, $\vec{v_h}$, $\mu_h$. This leads to the negative square terms in \eqref{eq:p2:energy_dissipation_discrete_stable}. Note that a similar  choice is not possible for $F_4, F_5$, $F_6$ since all of them 
comprise time derivatives. 
\qed
\end{proof}
%
%
%
%
\subsection{A Time Semi-Discrete Scheme for the NSAC System}\label{sec:p2:temporal}
As the  next step  we devise an energy-stable time discretization for the mixed formulation \eqref{eq:p2:mixed} of the NSAC system.  
Let $0=t_0 < t_1 < \ldots < t_N = T$ be a partition of $[0,T]$ with $N\in \mathbb{N}$. We set $\Delta t_n = t_{n+1}-t_n$. Moreover, we introduce for some function $\Phi:\Omega \to \mathbb{R}$ the notations 
\[
\Phi^n(\vec{x}) := \Phi(\vec{x},t_n)  \text{ and } \Phi^{n+\frac{1}{2}} := \frac{\Phi^{n+1}+\Phi^n}
{2} \quad (\vec{x} \in \Omega).
\]
The temporal discretisation is of Crank\---Nicholson type and chosen in a way that the discrete energy inequality holds. The resulting scheme is of second order accuracy.  

\begin{Definition}[Second-order  time discretization]\ \label{alg:p2:tempsemi}\\
Consider the initial conditions from  \eqref{eq:p2:IC}.
A function \[
\vec{U}^{\Delta t}:= 
(\rho^{\Delta t},\vec{v}^{\Delta t},\varphi^{\Delta t},\mu^{\Delta t},\tau^{\Delta t},\vec{\sigma}^{\Delta t}) 
: \Omega  \to   C^0([0,T])
\]
with 
\[
\vec{U}^{\Delta t}(t) =   \frac{1}{\Delta t_n}\big({\vec{U}^{n+1}-   \vec{U}^{n}}\big) (t- t_n) +   \vec{U}^{n}\quad (t \in [t_n,t_{n+1}], \, n\in\{0,\ldots,N-1\} )
\]
that satisfies $\vec{U}^{\Delta t}(t) \in {\mathcal V}_h \cap C^1(\bar \Omega)$ for $t\in [0,T]$
is called time-discrete solution of \eqref{eq:p2:mixed} if the time iterates  $\rho^{n+1},\vec{v}^{n+1},\varphi^{n+1},\mu^{n+1},\tau^{n+1},\vec{\sigma}^{n+1}$ satisfy  for $n \in \{0,\ldots, N-1\}$ the equations
    \begin{equation}\label{eq:p2:tempsemi}
    \begin{array}{rcl}
        0 &= & \ds   \frac{\rho^{n+1}-\rho^n}{\Delta t_n} + \div(\rho^{n+\frac{1}{2}}\vec{v}^{n+\frac{1}{2}}),\\[3ex]
        \vec{0}& =&  \ds \rho^{n+\frac{1}{2}}\left( \frac{\vec{v}^{n+1}-\vec{v}^n}{\Delta t_n} \right)
        + \div \big(\rho^{n+\frac{1}{2}}\vec{v}^{n+\frac{1}{2}} \otimes \vec{v}^{n+\frac{1}{2}}\big)\\[2ex]
        &&{} \ds  - \div(\rho^{n+\frac{1}{2}}\vec{v}^{n+\frac{1}{2}})\vec{v}^{n+\frac{1}{2}}  
         - \frac{1}{2} \rho^{n+\frac{1}{2}} \nabla|\vec{v}^{n+\frac{1}{2}}|^2 \\[2ex]  && {} \ds
        + \rho^{n+\frac{1}{2}} \nabla \tau^{n+\frac{1}{2}}
        - \mu^{n+\frac{1}{2}} \nabla \varphi^{n+\frac{1}{2}}
        - \div(\vec{S}\big(\varphi^{n+\frac{1}{2}},\nabla \vec{v}^{n+\frac{1}{2}}))\big),  \\[3ex]
        0 &=&{} \ds \frac{\varphi^{n+1}-\varphi^n}{\Delta t_n} + \nabla \varphi^{n+\frac{1}{2}}\cdot
        \vec{v}^{n+\frac{1}{2}} + \eta \frac{\mu^{n+\frac{1}{2}}}{\rho^{n+\frac{1}{2}}}, \\[4ex]
        0 &=& \mu^{n+\frac{1}{2}}
        - \frac{\rho^{n+1}  \tilde{f}(\rho^{n+1},\varphi^{n+1})- \rho^{n+1} \tilde{f}(\rho^{n+1},\varphi^{n}) +\tilde{f}(\rho^{n},\varphi^{n+1})- \rho^n\tilde{f}(\rho^n,\varphi^n)}{2(\varphi^{n+1}-\varphi^n)} \\[2ex]
        && + \gamma \div(\vec{\sigma}^{n+\frac{1}{2}}), \\[3ex]
        0& =& \tau^{n+\frac{1}{2}} -  \frac{\rho^{n+1} \tilde{f}(\rho^{n+1},\varphi^{n+1})-\rho^n \tilde{f}(\rho^{n},\varphi^{n+1})
        + \rho^{n+1} \tilde{f}(\rho^{n+1},\varphi^{n})-\rho^n \tilde{f}(\rho^n,\varphi^n)}{2(\rho^{n+1}-\rho^n)} \\[2ex]
        & &{}- \frac{1}{4} \left( |\vec{v}^{n+1}|^2 + |\vec{v}^n|^2\right), \\[3ex]
        \vec{0} &=& \vec{\sigma}^{n+1} - \nabla \varphi^{n+1}. 
    \end{array}
    \end{equation}
\end{Definition}
For the sake of simplicity we have chosen here the classical function space ${\mathcal V}_h \cap C^1(\bar \Omega)$  for the unknown  $\vec{U}^{\Delta t}$. We will return to weak formulations in Section \ref{sec:p2:fully_discrete}.
The  well-posedness of \eqref{eq:p2:tempsemi} could be assured for $\Delta t_n$ small enough and is supposed  to hold for the remainder of the section.
Different from the works \cite{Giesselmann2014}, 
a special feature of our time discretization is the staggered implicit-explicit treatment of the discrete  $\varphi$- and $\rho$-derivative of the mixture energy density $\rho \tilde f$ in $\eqref{eq:p2:tempsemi}_{4,5}$.
With this time discretization  we can provide unconditional energy stability for the dG method like it can be achieved by convex-concave splitting methods \cite{Eyre}.
The following theorem provides the exact statement for the  time-discrete solution  $\vec{U}^{\Delta t}$  of \eqref{eq:p2:mixed}.
\begin{thm}[Time-discrete energy stability] \label{thm:p2:energy_ineq_discrete}
    \ \\ 
    The  time-discrete solution  $\vec{U}^{\Delta t}$  of \eqref{eq:p2:mixed} as defined in Definition \ref{alg:p2:tempsemi} satisfies 
  for all $n \in \{0,\ldots,N\}$ the estimate
    \begin{equation}\label{timediscreteienergy}
    \begin{array}{rcl}
       \lefteqn{\int_\Omega \mkern-2mu\rho^n\tilde{f}(\rho^n,\varphi^n) + \frac{1}{2}\rho^n |\vec{v}^n|^2 + \frac{\gamma}{2} |\vec{\sigma}^n|^2 \ddd \vec{x}}\\[3ex]
        &=&{}  \ds \int_\Omega \mkern-2mu\rho^0\tilde{f}(\rho^0,\varphi^0) + \frac{1}{2}\rho^0 |\vec{v}^0|^2 + \frac{\gamma}{2} |\vec{\sigma}^0|^2 \dd \vec{x} \\[3ex]
        &&{}\ds - \sum_{j=0}^{n-1} \Delta t_j \int_\Omega \mkern-2mu \vec{S}(\varphi^{j+\frac{1}{2}},\nabla\vec{v}^{j+\frac{1}{2}}) : \nabla\vec{v}^{j+\frac{1}{2}}
        + \eta \frac{|\mu^{j+\frac{1}{2}}|^2}{\rho^{j+\frac{1}{2}}}  \dd\vec{x}.
    \end{array}
    \end{equation}
\end{thm}

\begin{proof} We define 
\begin{equation}\label{Tdefine}
\begin{array}{rcl}
T^{n+\frac12}_\mu&:=&\ds \rho^{n+1}\tilde{f}(\rho^{n+1},\varphi^{n+1})
        -\rho^{n+1} \tilde{f}(\rho^{n+1},\varphi^{n})\\[1.5ex]
 &&{}\ds        +\rho^{n}\tilde{f}(\rho^{n},\varphi^{n+1})
        -\rho^{n} \tilde{f}(\rho^{n},\varphi^{n})\\[2ex]
T^{n+\frac12}_\tau&:=&\ds\rho^{n+1} \tilde{f}(\rho^{n+1},\varphi^{n+1})
        -\rho^n \tilde{f}(\rho^n,\varphi^{n+1})\\[1.5ex]
        &&\ds {}
        +\rho^{n+1} \tilde{f}(\rho^{n+1},\varphi^{n})
        -\rho^n \tilde{f}(\rho^n,\varphi^n)
\end{array}
\end{equation}
    By multiplying $\eqref{eq:p2:tempsemi}_1$ with $\tau^{n+\frac{1}{2}}$, $\eqref{eq:p2:tempsemi}_2$ with $\vec{v}^{n+\frac{1}{2}}$ and $\eqref{eq:p2:tempsemi}_3$ with $\mu^{n+\frac{1}{2}}$ \, (, which is analogous to the choice of test functions in the proof of Theorem \ref{thm:p2:spatialdisc_conservation})
    we obtain after integrating over the domain $\Omega$ and using \eqref{Tdefine} the equations 
    \begin{align*}
        0 =& \int_\Omega \frac{1}{2\Delta t_n} T^{n+\frac12}_\mu  + \frac{\rho^{n+1}-\rho^n}{4\Delta t_n}\left( |\vec{v}^{n+1}|^2 + |\vec{v}^n|^2\right) \\
        &\hspace*{10em}+ \div(\rho^{n+\frac{1}{2}}\vec{v}^{n+\frac{1}{2}})\tau^{n+\frac{1}{2}} \ddd\vec{x}, \\
        0 =& \int_\Omega \frac{\rho^{n+\frac{1}{2}}}{2\Delta t_n}\left( |\vec{v}^{n+1}|^2- |\vec{v}^n|^2 \right)
        + \vec{v}^{n+\frac{1}{2}} \div(\rho^{n+\frac{1}{2}}\vec{v}^{n+\frac{1}{2}}\otimes \vec{v}^{n+\frac{1}{2}})
        \\
        &\ds -\div(\rho^{n+\frac{1}{2}}\vec{v}^{n+\frac{1}{2}})|\vec{v}^{n+\frac{1}{2}}|^2
        -\frac{1}{2}\rho^{n+\frac{1}{2}}\nabla|\vec{v}^{n+\frac{1}{2}}|^2 \vec{v}^{n+\frac{1}{2}}
        + \rho^{n+\frac{1}{2}}\vec{v}^{n+\frac{1}{2}} \nabla \tau^{n+\frac{1}{2}} \\
        &- \mu^{n+\frac{1}{2}}\nabla \varphi^{n+\frac{1}{2}}\vec{v}^{n+\frac{1}{2}}
        -\div(\vec{S}(\varphi^{n+\frac{1}{2}},\nabla \vec{v}^{n+\frac{1}{2}}))\cdot \vec{v}^{n+\frac{1}{2}} \ddd\vec{x}, \\
        0 =& \int_\Omega \frac{1}{2\Delta t_n} T^{n+\frac12}_\tau - \frac{\varphi^{n+1}-\varphi^n}{\Delta t_n} \gamma \div(\vec{\sigma}^{n+\frac{1}{2}})\\
    &     \hspace*{10em}   + \nabla \varphi^{n+\frac{1}{2}}\cdot \vec{v}^{n+\frac{1}{2}}\mu^{n+\frac{1}{2}}
        + \eta \frac{|\mu^{n+\frac{1}{2}}|^2}{\rho^{n+\frac{1}{2}}} \ddd \vec{x}.
    \end{align*}
If we add up the latter expressions and use the  definitions 
    \begin{equation*}
    \begin{array}{rcl}
        I^{n+1}_1 &:=
         &\ds \frac{\rho^{n+1}-\rho^n}{\Delta t_n}\left(\frac{T^{n+\frac12}_\mu}{2(\rho^{n+1}-\rho^n)} + \frac{1}{4} \left( |\vec{v}^{n+1}|^2 + |\vec{v}^n|^2 \right) \right)\\[2ex]
        && \ds {}+ \frac{1}{\Delta t_n} \rho^{n+\frac{1}{2}}\vec{v}^{n+\frac{1}{2}}\cdot(\vec{v}^{n+1}-\vec{v}^n)   \\[2ex]
        && \ds {}+\frac{\varphi^{n+1}-\varphi^n}{\Delta t_n}  \left(\frac{T^{n+\frac12}_\tau}{2(\varphi^{n+1}-\varphi^n)}  - \gamma \div(\vec{\sigma}^{n+\frac{1}{2}})\right), \\[3ex]
        I^{n+1}_2 &:=& \div\big(\rho^{n+\frac{1}{2}}\vec{v}^{n+\frac{1}{2}}\big) \tau^{n+\frac{1}{2}} +\rho^{n+\frac{1}{2}}\vec{v}^{n+\frac{1}{2}}\nabla \tau^{n+\frac{1}{2}}, \\[3ex]
        I^{n+1}_3 &:=& \vec{v}^{n+\frac{1}{2}}
        \div(\rho^{n+\frac{1}{2}}\vec{v}^{n+\frac{1}{2}}\otimes\vec{v}^{n+\frac{1}{2}})
        -\div(\rho^{n+\frac{1}{2}}\vec{v}^{n+\frac{1}{2}})|\vec{v}^{n+\frac{1}{2}}|^2\\[2ex]
        &&\ds{}- \frac{1}{2}\rho^{n+\frac{1}{2}}\nabla|\vec{v}^{n+\frac{1}{2}}|^2\vec{v}^{n+\frac{1}{2}}, \\
        I^{n+1}_4 &:=&\ds  - \div(\vec{S}(\varphi^{n+\frac{1}{2}},\vec{v}^{n+\frac{1}{2}}) \cdot \vec{v}^{n+\frac{1}{2}}
        + \eta \frac{|\mu^{n+\frac{1}{2}}|^2}{\rho^{n+\frac{1}{2}}},
    \end{array}
 \end{equation*}
 we get after rearrangement of terms 
 \begin{equation}\label{sumzero}
 \int_\Omega I^{n+1}_1 + I^{n+1}_2 +I^{n+1}_3+I^{n+1}_4 \dd \vec{x} =0.
 \end{equation}
 Note that the boundary condition on $\vec{v}$ implies $I^{n+1}_3\equiv 0$.
Furthermore, we compute with $\eqref{eq:p2:tempsemi}_{4,5,6}$, the boundary conditions, and using the definitions of the terms $T^{n+\frac12}_\mu,\,T^{n+\frac12}_\tau$  from \eqref{Tdefine} the relations
    \begin{align*}
        \Delta t_n \int_\Omega I^{n+1}_1 \dd\vec{x} =& \int_\Omega \rho^{n+1} \tilde{f}(\rho^{n+1},\varphi^{n+1}) + \frac{1}{2}\rho^{n+1}|\vec{v}^{n+1}|^2 + \frac{\gamma}{2}|\vec{\sigma}^{n+1}|^2 \ddd\vec{x} \\
        &- \int_\Omega \rho^n \tilde{f}(\rho^n,\varphi^n) + \frac{1}{2}\rho^n |\vec{v}^n|^2 + \frac{\gamma}{2} |\vec{\sigma}^n|^2 \ddd\vec{x},\\[2ex]
        \int_\Omega I^{n+1}_2 \ddd\vec{x} =&\int_\Omega \div\big(\rho^{n+\frac{1}{2}}\vec{v}^{n+\frac{1}{2}}\tau^{n+\frac{1}{2}}\big) \dd \vec{x}  =   0,\\[2ex]
        \int_\Omega I^{n+1}_4 \ddd \vec{x} =& \int_\Omega \vec{S}(\varphi^{n+\frac{1}{2}},\vec{v}^{n+\frac{1}{2}}) : \nabla \vec{v}^{n+\frac{1}{2}} + \eta \frac{|\mu^{n+\frac{1}{2}}|^2}{\rho^{n+\frac{1}{2}}} \ddd \vec{x}.
    \end{align*}
    Because of \eqref{sumzero} the assertion is shown after adding with respect to $n=0,\ldots,N-1$.\qed
\end{proof}

\subsection{A Fully-Discrete dG Method for the NSAC System}\label{sec:p2:fully_discrete}
Finally, in this section we present
the fully-discrete dG method for \eqref{eq:p2:mixed}.
The approach  combines the two methods presented in the 
previous Sections \ref{sec:p2:spatial} and \ref{sec:p2:temporal}. As in Section \ref{sec:p2:spatial}
we consider the function space ${\mathcal V}_h$ on a conforming triangulation $\mathcal{T}$ of the domain $\Omega$, see  \eqref{functionspace}. Again,  we  decompose 
the time interval according to $0 = t_0 < t_1 < \ldots < t_N = T$, and denote $\Delta t_n = t_{n+1}-t_n$.
%
%
\begin{Definition}[Fully-discrete 
 dG method]\label{alg:p2:fully_discrete} \ \\
Consider initial functions from \eqref{eq:p2:IC} and \eqref{iniremainder} such
that $\vec{U}_{0,h} = \Pi_h \vec{U}_0 \in {\mathcal V}_h$ holds.  
A function
\[
\vec{U}^{\Delta t}_h:=
(\rho^{\Delta t}_h,\vec{v}^{\Delta t}_h,\varphi^{\Delta t}_h,\mu^{\Delta t}_h,\tau^{\Delta t}_h,\vec{\sigma}^{\Delta t}_h) \in   C^0([0,T]; {\mathcal V}_h)
\]
with $\vec{U}^{\Delta t}_h(\cdot,0) = \vec{U}_{0,h}$ and 
\[
\begin{array}{rcl}
\vec{U}^{\Delta t}_h(\vec{x},t) &=& \ds   \frac{1}{\Delta t_n}\big({\vec{U}^{n+1}_h(\vec{x})-   \vec{U}^{n}_h}(\vec{x})\big) (t- t_n) +   \vec{U}^{n}_h(\vec{x})\\[1.9ex]
&&\quad (\vec{x}\in \Omega,\, t \in [t_n,t_{n+1}], \, n\in\{0,\ldots,N-1\} ),
\end{array}
\]
that obeys to
\begin{equation}\label{eq:reguhfull}
  \begin{array}{c}\rho^{\Delta t}_h > 0,\,   \varphi^{\Delta t}_h \in [0,1]\text{ a.e. in } \Omega \times [0,T],\\[1.1ex]
  ({\rho^{\Delta t}_h)}^{-1}, \,  (\rho^{\Delta t}_h \tilde f(\varphi^{\Delta t}_h,\rho^{\Delta t}_h)), \,   \rho^{\Delta t}_h {|\vec{v}^{\Delta t}_h|}^2  \in C^0([0,T];L^2(\Omega)) 
  \end{array}
  \end{equation}
is called  a fully-discrete dG approximation  of the weak solution $\vec{U} \in {\mathcal V}$ of  \eqref{eq:p2:mixed} if the time iterate  
$\vec{U}^{n+1}_h ;=(\rho_h^{n+1}, \vec{v}_h^{n+1}, \varphi_h^{n+1},$ $\mu_h^{n+1},\tau_h^{n+1}, \vec{\sigma}_h^{n+1}) \in \mathcal{V}_h$ satisfies
        \begin{equation}\label{eq:p2:disc}
        \begin{array}{rcl}
            0 &=& \ds \int_\Omega \left(\frac{\rho_h^{n+1}-\rho_h^n}{\Delta t}
            + \div(\rho_h^{n+\frac{1}{2}}\vec{v}_h^{n+\frac{1}{2}})\right) \psi_h \dd \vec{x}
            - \int_{\mathcal{E}} \jump{\rho_h^{n+\frac{1}{2}}\vec{v}_h^{n+\frac{1}{2}}} \avg{\psi_h} \dd s,\\[3ex]
            0 &=& \ds \int_\Omega \left(\rho_h^{n+\frac{1}{2}}\left(\frac{\vec{v}_h^{n+1}-\vec{v}_h^n}{\Delta t}\right) + \div(\rho_h^{n+\frac{1}{2}}\vec{v}_h^{n+\frac{1}{2}}\otimes
            \vec{v}_h^{n+\frac{1}{2}})\right.\\[2ex]
            &&\ds  {}- \div(\rho_h^{n+\frac{1}{2}}\vec{v}_h^{n+\frac{1}{2}})\vec{v}_h^{n+\frac{1}{2}}  \\[2ex]
            &&\ds {}\left. - \frac{1}{2} \rho_h^{n+\frac{1}{2}}\nabla|\vec{v}_h^{n+\frac{1}{2}}|^2 +
            \rho_h^{n+\frac{1}{2}}\nabla \tau_h^{n+\frac{1}{2}}
            - \mu_h^{n+\frac{1}{2}}\nabla \varphi_h^{n+\frac{1}{2}}\right)\cdot \vec{X}_h \dd \vec{x}\\[3ex]
            & &\ds -\int_{\mathcal{E}} \jump{\tau_h^{n+\frac{1}{2}}}\cdot\avg{\rho_h^{n+\frac{1}{2}}\vec{X}_h}
            - \jump{\varphi_h^{n+\frac{1}{2}}}\cdot\avg{\mu_h^{n+\frac{1}{2}}\vec{X}_h} \dd s\\[2ex]
            &&{}\ds+ B_h\big[\varphi_h^{n+\frac{1}{2}};\vec{v}_h^{n+\frac{1}{2}},\vec{X}_h\big],\\[3ex] 
            0&=& \ds \int_\Omega \bigg(\frac{\varphi_h^{n+1}-\varphi_h^n}{\Delta t}
            + \nabla \varphi_h^{n+\frac{1}{2}}\cdot \vec{v}_h^{n+\frac{1}{2}}
            + \eta\frac{\mu_h^{n+\frac{1}{2}}}{\rho_h^{n+\frac{1}{2}}}\bigg) \Theta_h \dd \vec{x}\\[2ex]
            &&{}\ds- \int_{\mathcal{E}} \jump{\varphi_h^{n+\frac{1}{2}}}\cdot\avg{\Theta_h\vec{v}_h^{n+\frac{1}{2}}} \dd s,\\[3ex]
            0&=& \ds \int_\Omega \left(\mu_h^{n+\frac{1}{2}} - \frac{\rho^{n+1} \tilde{f}(\rho^{n+1},\varphi^{n+1})
            -\rho^{n+1} \tilde{f}(\rho^{n+1},\varphi^{n})}{2(\varphi^{n+1}-\varphi^n)} \right.  \\[2ex]
            && \ds \hspace*{4.7em}- \frac{\rho^n \tilde{f}(\rho^{n},\varphi^{n+1})-\rho^n \tilde{f}(\rho^n,\varphi^n)}{2(\varphi^{n+1}-\varphi^n)},\\[2ex] &&\ds \hspace*{6em}\left.\phantom{\int_\Omega} + \gamma \div(\vec{\sigma}_h^{n+\frac{1}{2}})\right)  \chi_h \dd \vec{x} - \int_{\mathcal{E}} \gamma \jump{\vec{\sigma}_h^{n+\frac{1}{2}}}\avg{\chi_h} \dd s
       \end{array} 
       \end{equation}
       \begin{equation*}
       \begin{array}{rcl}
            0&=&\ds \int_\Omega \left( \tau_h^{n+\frac{1}{2}} - \frac{\rho^{n+1} \tilde{f}(\rho^{n+1},\varphi^{n+1})-\rho^n \tilde{f}(\rho^{n},\varphi^{n+1})}{2(\rho^{n+1}-\rho^n)} \right.  \\[2ex]
            &&\ds {}  \hspace*{4.5em}-\frac{\rho^{n+1} \tilde{f}(\rho^{n+1},\varphi^{n})-\rho^n \tilde{f}(\rho^n,\varphi^n)}{2(\rho^{n+1}-\rho^n)} \\[2ex]
            &&\ds{} \left.  \hspace*{10em}-\frac{1}{4}(|\vec{v}_h^{n+1}|^2+|\vec{v}_h^{n}|^2)\right) \zeta_h \dd \vec{x}, \\[3ex]
            0&=& \ds \int_\Omega \left(\vec{\sigma}_h^{n+1} - \nabla \varphi_h^{n+1}\right) \cdot \vec{Z}_h \ddd \vec{x}
            + \int_{\mathcal{E}} \jump{\varphi_h^{n+1}}\cdot\avg{\vec{Z}_h} \dd s
        \end{array}
        \end{equation*}
     for all $n \in \{ 0,\ldots, N-1\}$ and for all test functions  $(\psi_h, \vec{X}_h, \Theta_h, \chi_h, \zeta_h, \vec{Z}_h) \in \mathcal{V}_h$.
\end{Definition}

%
Note that the formulation \eqref{eq:p2:disc} relies exactly on the choice \eqref{eq:fluxchoice} for the fluxes $F_1,\ldots,F_6$. One can of course also choose the stabilized fluxes as in Corollary \ref{cor:stable}.

The discretization is chosen such that the discrete counterpart \ref{thm:p2:discrete_energy_ineq} of the energy inequality is satisfied.

\begin{thm}[Fully-discrete energy stability ]\label{thm:p2:discrete_energy_ineq}
Let $\vec{U}^{\Delta t}_h \in   C^0([0,T]; {\mathcal V}_h)$
 be a fully-discrete dG approximation  of the weak solution $\vec{U} \in  C^0([0,T];{\mathcal V})$ of  \eqref{eq:p2:mixed}.\\
 Then we  have
 \begin{enumerate}
 \item  mass conservation, i.e.
  \[
  \int_\Omega  \rho^{\Delta t}_h(\vec{x},t) \dd \vec{x} =  \int_\Omega  \rho_0(\vec{x}) \dd \vec{x} 
  \]
  for all $t\in [0,T]$ and
  \item the entropy dissipation inequality 
    \begin{equation}\label{energy:fulldiscrete}
    \begin{array}{rcl}
      \hat E^{\Delta t}_h(t_{n+1}) &\le& \hat E^{\Delta t}_h(t_{n})\\
      && \ds {}- \Delta t\int_\Omega \eta\frac{|\mu_h^{n+\frac{1}{2}}|^2}{\rho_h^{n+\frac{1}{2}}} \dd \vec{x}- \Delta t   B_h\Big[\varphi_h^{n+\frac{1}{2}};\vec{v}_h^{n+\frac{1}{2}},\vec{v}_h^{n+\frac{1}{2}}\Big]
    \end{array}
    \end{equation}
    for all $n\in \{0,\ldots,N-1\}$. Thereby we used
    \begin{align*}
     \tilde   E^{\Delta t}_h(t_{n}) = \int_\Omega \rho_h^n\tilde{f}(\rho_h^{n},\varphi_h^{n})+ \frac{\gamma}{2}|\vec{\sigma}_h^{n}|^2 + \frac{\rho_h^{n}}{2}|\vec{v}_h^{n}|^2 &\dd \vec{x}.     
    \end{align*}
\end{enumerate}
\end{thm}

\begin{proof}
    The proof  follows from combining the proofs of Theorem \ref{thm:p2:spatialdisc_conservation} and Theorem \ref{thm:p2:energy_ineq_discrete}.\qed
\end{proof}

\section{Numerical Experiments}\label{sec:NumExp}
For the equations of state in the bulk phases, we choose stiffened gas equations
\begin{align}\label{eq:p2:stiffened_gas}
  \rho f_\mathrm{L/V} = \alpha_\mathrm{L/V} \rho \ln(\rho) + (\beta_\mathrm{L/V}-\alpha_\mathrm{L/V})\rho + \gamma_\mathrm{L/V},
\end{align}
with parameters
\begin{alignat}{4}
    \alpha_\mathrm{L} &= 1.5, \quad &\alpha_\mathrm{V} &= 1,\\
    \beta_\mathrm{L}  &= \ln(2), &\beta_\mathrm{V}  &= 0, \\
    \gamma_\mathrm{L} &= 0,      &\gamma_\mathrm{V} &= 0.5.
\end{alignat}

For the bulk viscosities we set $\mu_\mathrm{L}=\mu_\mathrm{V}=0.001$. The capillary parameter is taken $\gamma = 0.001$ and the mobility $\eta = 1$.
The double well $W$ from \eqref{eq:p2:def_h} is chosen  with  $a=0.1$.\\
We implemented the energy-stable dG method from Definition  \ref{alg:p2:fully_discrete} with the finite
element toolbox FEniCS, which is based on the C++ library DOLFIN \cite{Fenics}. In each time step the nonlinear system is solved by an inexact Newton method. The linear subsystems are approximated by a biconjugate gradient stabilized method (bicgstab) with an incomplete LU preconditioner. The absolute tolerance of both solvers is set to $10^{-10}$.
\subsection{Convergence Studies}

In this section, we conduct numerical experiments to verify the convergence rate of the fully discrete dG method (see equation \eqref{alg:p2:fully_discrete}) with respect to space and time.
We test the convergence properties in one spatial dimension on $\Omega  = [0, 1]$. As there is no known explicit (non-trivial) solution to equation \eqref{eq:p2:mixed}, we use a manufactured solution.
Define
\begin{equation} \label{eq:p2:manufactured_sol}
\begin{array}{c}
\ds    \rho(x,t) := \frac{1}{2}\cos(5\pi t)\cos(2\pi x) + \frac{3}{2}, \quad
    v(x,t) = \cos(5\pi t)\cos(4\pi x), \\[1.8ex]
    \ds
    \varphi(x,t) := \frac{1}{2}\cos(5\pi t)\cos(2\pi x) + \frac{1}{2}.
    \end{array}
\end{equation}
These functions satisfy the NSAC system {\eqref{eq:p2:NSAC} exactly if we 
add source terms $S_{\rho}(x,t), S_v(x,t), $ and $S_{\varphi}(x,t)$ for ${\eqref{eq:p2:NSAC}}$. We used the Sympy python library \cite{Meurer2017} for the symbolical calculations leading to the source terms.

\subsubsection{Order of Convergence in Space}

For the computations we select a time step size depending on the dG polynomial degree.
For $k \in \{0,1\}$ we choose $\Delta t = 10^{\lfloor \log_{10}(\sfrac{1}{N})\rfloor},$ where $N$ denotes the number of cells. For $k\geq 2$ we choose $\Delta t = 10^{\lfloor \log_{10}(\sfrac{1}{N^2})\rfloor}$.

We run the simulation up to $T=0.03$ on grids with different number $N$ of cells.
We found that the convergence rate depends on the choice of stabilization parameters $\alpha_B$ and $\alpha_1$, see \eqref{eq:bilinearform}, \eqref{eq:p2:flux1}. This is especially true for $k=1$.
We list the parameters we used in \autoref{tab:p2:stab_parameter}.
The other parameters $\alpha_2,\, \alpha_3$ in \eqref{eq:p2:flux2}, \eqref{eq:p2:flux3} are chosen to be $0$.

\begin{table}[ht]
    \centering
    \begin{tabular}{ccc}
        \hline
        $k$ & $\alpha_B$    & $\alpha_1$ \\ \hline
        $0$ & 1e-03         & 0           \\
        $1$ & 1.7e-03       & 6e-03       \\
        $2$ & 7e-03         & 1e-03       \\
        $3$ & 2e-02         & 1e-01 \\\hline
    \end{tabular}
    \caption{Numerical stabilization parameters used in the numerical experiments, depending on polynomial degree $k$.}
    \label{tab:p2:stab_parameter}
\end{table}

We investigate the errors of the discrete solution $(\rho_h,v_h,\varphi_h)$ and the exact solution $(\rho,v,\allowbreak \varphi)$ from \eqref{eq:p2:manufactured_sol} in $L^{\infty}(0, T ; L^2(\Omega))$.
The results in Tables \ref{tab:p2:k0}--\ref{tab:p2:k3} indicate that the scheme converges with order $k+1$ in space and order $2$ in time.

\begin{table}[!ht]
    \centering
    \begin{small}
        \begin{tabular}{ccccccc}
            \hline
             $N$ & $\|\rho-\rho_h\|_{L^{\infty}(L^2)}$ & $\mathrm{EOC}_{\rho}$
             & $\|v-v_h\|_{L^{\infty}(L^2)}$ & $\mathrm{EOC}_{v}$
             & $\|\varphi-\varphi_h\|_{L^{\infty}(L^2)}$ & $\mathrm{EOC}_{\varphi}$  
             \\ \hline
            \begin{filecontents*}{./Graphics/P0_error.csv}
,Size,errorRho,errorV,errorPhi,errorMu,errorTau,errorSigma,EOCrho,EOCv,EOCphi,EOCmu,EOCtau,EOCsigma
0,16,5.3927e-02,1.6099e-01,3.9793e-02,3.0310e+00,1.2293e-01,3.2093e-01,--,--,--,--,--,--
1,32,2.3266e-02,7.9256e-02,1.9836e-02,1.5238e+00,5.3727e-02,1.3338e-01,1.21,1.02,1.00,0.99,1.19,1.27
2,64,1.0400e-02,3.9587e-02,9.9282e-03,7.6867e-01,2.6745e-02,6.2387e-02,1.16,1.00,1.00,0.99,1.01,1.10
3,128,5.0357e-03,2.0036e-02,5.0093e-03,3.9055e-01,1.3561e-02,3.1577e-02,1.05,0.98,0.99,0.98,0.98,0.98
4,256,2.5047e-03,1.0019e-02,2.5047e-03,1.9529e-01,6.7813e-03,1.5761e-02,1.01,1.00,1.00,1.00,1.00,1.00
5,512,1.2523e-03,5.0093e-03,1.2523e-03,9.7650e-02,3.3907e-03,7.8737e-03,1.00,1.00,1.00,1.00,1.00,1.00
6,1024,6.2624e-04,2.5050e-03,6.2624e-04,4.8838e-02,1.6957e-03,3.9366e-03,1.00,1.00,1.00,1.00,1.00,1.00
7,2048,3.1312e-04,1.2525e-03,3.1312e-04,2.4419e-02,8.4783e-04,1.9682e-03,1.00,1.00,1.00,1.00,1.00,1.00
8,4096,1.5656e-04,6.2624e-04,1.5656e-04,1.2210e-02,4.2392e-04,9.8398e-04,1.00,1.00,1.00,1.00,1.00,1.00
9,8192,7.8281e-05,3.1312e-04,7.8281e-05,6.1048e-03,2.1196e-04,4.9192e-04,1.00,1.00,1.00,1.00,1.00,1.00
\end{filecontents*}
\csvreader[head to column names,
                       late after line=\\]{./Graphics/P0_error.csv}{}%
            {\Size & \errorRho & \EOCrho  & \errorV & \EOCv  & \errorPhi & \EOCphi}%
             \hline
         \end{tabular}
     \end{small}
     \caption{Errors in the $L^{\infty}(0, T ; L^2(\Omega))$-norm for dG polynomial degree  $k=0$.}
     \label{tab:p2:k0}
\end{table}

\begin{table}[!ht]
    \centering
   \begin{small}
    \begin{tabular}{@{\!}ccccccc@{\!}}
        \hline
         $N$ & $\|\rho-\rho_h\|_{L^{\infty}(L^2)}$ & $\mathrm{EOC}_{\rho}$
         & $\|v-v_h\|_{L^{\infty}(L^2)}$ & $\mathrm{EOC}_{v}$
         & $\|\varphi-\varphi_h\|_{L^{\infty}(L^2)}$ & $\mathrm{EOC}_{\varphi}$  
         \\ \hline
        \begin{filecontents*}{./Graphics/P1_error.csv}
,Size,errorRho,errorV,errorPhi,errorMu,errorTau,errorSigma,EOCrho,EOCv,EOCphi,EOCmu,EOCtau,EOCsigma
0,16,7.0177e-02,8.4485e-02,1.4141e-02,1.4945e+00,1.0457e-01,3.8855e-01,0.00,0.00,0.00,0.00,0.00,0.00
1,32,3.4156e-02,3.5802e-02,4.8521e-03,4.5764e-01,4.5995e-02,2.2237e-01,1.04,1.24,1.54,1.71,1.18,0.81
2,64,1.7161e-02,1.4091e-02,2.5639e-03,2.1821e-01,2.0222e-02,1.6442e-01,0.99,1.35,0.92,1.07,1.19,0.44
3,128,8.9180e-03,3.6522e-03,5.0880e-04,3.6550e-02,9.0830e-03,6.6462e-02,0.94,1.95,2.33,2.58,1.15,1.31
4,256,4.1448e-03,6.8157e-04,2.3976e-04,1.5408e-02,3.6193e-03,2.7911e-02,1.11,2.42,1.09,1.25,1.33,1.25
5,512,1.6387e-03,1.9353e-04,1.0784e-04,7.5878e-03,1.3495e-03,1.1834e-02,1.34,1.82,1.15,1.02,1.42,1.24
6,1024,5.8990e-04,2.5084e-05,3.5755e-05,2.5216e-03,4.7968e-04,4.5100e-03,1.47,2.95,1.59,1.59,1.49,1.39
7,2048,1.7574e-04,3.4145e-06,1.0780e-05,1.2662e-03,1.4268e-04,1.9674e-03,1.75,2.88,1.73,0.99,1.75,1.20
\end{filecontents*}
\csvreader[head to column names,late after line=\\]{./Graphics/P1_error.csv}{}%
        {\Size & \errorRho & \EOCrho  & \errorV & \EOCv  & \errorPhi & \EOCphi }%
         \hline
     \end{tabular}
     \end{small}
     \caption{Errors in the $L^{\infty}(0, T ; L^2(\Omega))$-norm for dG polynomial degree $k=1$.}
\label{tab:p2:k1}
\end{table}

\begin{table}[!ht]
    \centering
    \begin{small}
        \begin{tabular}{@{\!}ccccccc@{\!}}
            \hline
             $N$ & $\|\rho-\rho_h\|_{L^{\infty}(L^2)}$ & $\mathrm{EOC}_{\rho}$
             & $\|v-v_h\|_{L^{\infty}(L^2)}$ & $\mathrm{EOC}_{v}$
             & $\|\varphi-\varphi_h\|_{L^{\infty}(L^2)}$ & $\mathrm{EOC}_{\varphi}$  
             \\ \hline
            \begin{filecontents*}{./Graphics/P2_error.csv}
,Size,errorRho,errorV,errorPhi,errorMu,errorTau,errorSigma,EOCrho,EOCv,EOCphi,EOCmu,EOCtau,EOCsigma
0,16,9.4031e-03,5.0512e-03,1.0466e-03,8.5868e-02,8.9875e-03,1.0429e-01,--,--,--,--,--,--
1,32,1.3901e-03,5.2801e-04,1.3235e-04,1.1112e-02,1.3953e-03,2.4489e-02,2.76,3.26,2.98,2.95,2.69,2.09
2,64,3.4729e-05,3.0377e-05,2.6400e-06,1.3201e-03,6.2876e-05,7.9256e-04,5.32,4.12,5.65,3.07,4.47,4.95
3,128,3.3313e-06,3.8284e-06,2.3993e-07,2.4896e-04,7.5523e-06,1.9884e-04,3.38,2.99,3.46,2.41,3.06,1.99
4,256,3.2122e-07,4.7174e-07,2.9784e-08,9.8483e-05,9.3751e-07,4.9216e-05,3.37,3.02,3.01,1.34,3.01,2.01
5,512,2.9897e-08,5.9655e-08,3.7434e-09,4.9152e-05,1.1789e-07,1.2401e-05,3.43,2.98,2.99,1.00,2.99,1.99
\end{filecontents*}
\csvreader[head to column names,late after line=\\]{./Graphics/P2_error.csv}{}%
            {\Size & \errorRho & \EOCrho  & \errorV & \EOCv  & \errorPhi & \EOCphi }%
             \hline
         \end{tabular}
     \end{small}
     \caption{Errors in the $L^{\infty}(0,T;L^2(\Omega))$-norm for dG polynomial degree $k=2$.}
\label{tab:p2:k2}
\end{table}

\begin{table}[!ht]
    \centering
    \begin{small}
        \begin{tabular}{@{\!}ccccccc@{\!}}
             \hline
             $N$ & $\|\rho-\rho_h\|_{L^{\infty}(L^2)}$ & $\mathrm{EOC}_{\rho}$
             & $\|v-v_h\|_{L^{\infty}(L^2)}$ & $\mathrm{EOC}_{v}$
             & $\|\varphi-\varphi_h\|_{L^{\infty}(L^2)}$ & $\mathrm{EOC}_{\varphi}$  
             \\ \hline
            \begin{filecontents*}{./Graphics/P3_error.csv}
,Size,errorRho,errorV,errorPhi,errorMu,errorTau,errorSigma,EOCrho,EOCv,EOCphi,EOCmu,EOCtau,EOCsigma
0,16,1.1536e-03,6.7507e-04,7.9957e-05,6.4400e-03,1.2978e-03,1.2231e-02,--,--,--,--,--,--
1,32,5.2932e-05,1.6833e-05,3.7722e-06,5.7119e-04,5.2053e-05,1.1373e-03,4.45,5.33,4.41,3.50,4.64,3.43
2,64,2.5162e-06,8.9823e-07,2.3245e-07,2.9405e-05,2.6150e-06,4.6917e-05,4.39,4.23,4.02,4.28,4.32,4.60
3,128,1.4674e-07,3.3460e-08,1.3706e-08,3.2861e-06,1.4381e-07,6.8058e-06,4.10,4.75,4.08,3.16,4.18,2.79
\end{filecontents*}
\csvreader[head to column names,late after line=\\]{./Graphics/P3_error.csv}{}%
            {\Size & \errorRho & \EOCrho  & \errorV & \EOCv  & \errorPhi & \EOCphi }%
             \hline
         \end{tabular}
     \end{small}
     \caption{Errors in the $L^{\infty}(0,T;L^2(\Omega))$-norm for dG polynomial degree $k=3$.}
\label{tab:p2:k3}
\end{table}

\subsubsection{Order of Convergence in Time}
Since the convergence order for $k=1$ is not evident from the previous example, we present  an additional experiment. It investigates the convergence order with respect to time. In order to ensure a low spatial discretization error we use dG polynomials with  high degree $k=5$ for different $\Delta t$ on a mesh with $64$ cells.
For the numerical simulation we use $\alpha_B = \alpha_1 = 1.5$.
We analyze  the errors of the discrete solution $(\rho_h,v_h,\varphi_h)$ and the exact solution $(\rho,v,\varphi)$  from \eqref{eq:p2:manufactured_sol} in the $L^2(\Omega)$-norm at final time $T=0.03$. The results in \autoref{tab:p2:k5} show clearly  that the scheme converges with order two  in time.

\begin{table}[!ht]
    \centering
    \begin{small}
        \begin{tabular}{ccccccc}
            \hline
             $\Delta t$ & $\|\rho-\rho_h\|_{L^2}$ & $\mathrm{EOC}_{\rho}$
             & $\|v-v_h\|_{L^2}$ & $\mathrm{EOC}_{v}$
             & $\|\varphi-\varphi_h\|_{L^2}$ & $\mathrm{EOC}_{\varphi}$  
             \\ \hline
          \begin{filecontents*}{./Graphics/P5_error.csv}
,dt,errorRho,errorV,errorPhi,errorMu,errorTau,errorSigma,EOCrho,EOCv,EOCphi,EOCmu,EOCtau,EOCsigma
0,1.00e-02,2.6833e-03,3.8994e-03,1.6787e-03,1.3628e-01,3.5347e-03,3.7373e-02,-,-,-,-,-,-
1,5.00e-03,6.4094e-04,9.3611e-04,3.9728e-04,3.1599e-02,8.5994e-04,8.8980e-03,2.07,2.06,2.08,2.11,2.04,2.07
2,1.00e-03,2.5357e-05,3.7026e-05,1.5626e-05,1.2361e-03,3.4156e-05,3.5049e-04,2.01,2.01,2.01,2.01,2.00,2.01
3,5.00e-04,6.3372e-06,9.2533e-06,3.9046e-06,3.0882e-04,8.5374e-06,8.7583e-05,2.00,2.00,2.00,2.00,2.00,2.00
4,1.00e-04,2.5348e-07,3.7009e-07,1.5616e-07,1.2350e-05,3.4157e-07,3.5032e-06,2.00,2.00,2.00,2.00,2.00,2.00
5,5.00e-05,6.3383e-08,9.2525e-08,3.9041e-08,3.0880e-06,8.5429e-08,8.7608e-07,2.00,2.00,2.00,2.00,2.00,2.00
6,1.00e-05,2.5908e-09,3.8099e-09,1.5698e-09,1.2470e-07,4.1053e-09,3.5588e-08,1.99,1.98,2.00,1.99,1.89,1.99
\end{filecontents*}
\csvreader[head to column names,late after line=\\]{./Graphics/P5_error.csv}{}%
           {\dt & \errorRho & \EOCrho  & \errorV & \EOCv  & \errorPhi & \EOCphi }%
             \hline
         \end{tabular}
     \end{small}
     \caption{Errors in the $L^2(\Omega)$-norm at end time  $T=0.03$ for dG polynomial degree $k=5$.}
\label{tab:p2:k5}
\end{table}

\subsection{Discrete Energy Dissipation}\label{sec:p2:ex1}
 
We consider 
in this section the example of two merging droplets in two dimensions.
We showcase  that phase-field models can handle topological changes. More important, we investigate the discrete energy and validate the statement  on the  energy dissipation from Theorem \ref{thm:p2:energy_ineq_discrete}. The results of this experiment have already been reported in the conference paper
\cite{Massa20}.

For the bulk viscosities we set $\mu_{\mathrm{L}} = 0.0125$ and $\mu_{\mathrm{V}} = 0.00125$. The
capillary parameter is taken
$\gamma =5\cdot 10^{-4}$
and the mobility $\eta = 10$. The polynomial
order of the dG polynomials is $k=2$.

For the EOS, we choose  the stiffened gas equation
\begin{align} 
  \rho f_\mathrm{L/V} = \alpha_\mathrm{L/V} \rho \ln(\rho) + (\beta_\mathrm{L/V}-\alpha_\mathrm{L/V})\rho + \gamma_\mathrm{L/V},
\end{align}
with parameters
\begin{alignat}{4}
    \alpha_\mathrm{L} &= 5, \quad &\alpha_\mathrm{V} &= 1.5,\\
    \beta_\mathrm{L}  &= -4,      &\beta_\mathrm{V}  &= 1.8, \\
    \gamma_\mathrm{L} &= 11,      &\gamma_\mathrm{V} &= 0.324.
\end{alignat}
Initially we have  static conditions, i.e., $\vec{v}_0 = \boldsymbol{0}$, and look at two neighboring droplets of different size. The computational domain is $\Omega = [0,1]\times[0,1]$. The droplets are located at $(0.39,0.5)$ and $(0.6,0.5)$ with radii $0.08$ and $0.12$. The initial density profile is smeared out with value $\rho_\mathrm{L}=2.23$ inside and $\rho_\mathrm{V} = 0.3$ outside the droplet. As expected the droplets join and after some time  form one larger droplet. This evolution with  mobility $\eta = 10$ is depicted in \autoref{fig:p2:merging}.\\
\begin{figure}[ht]
    \centering
    \includegraphics[width=.9\textwidth]{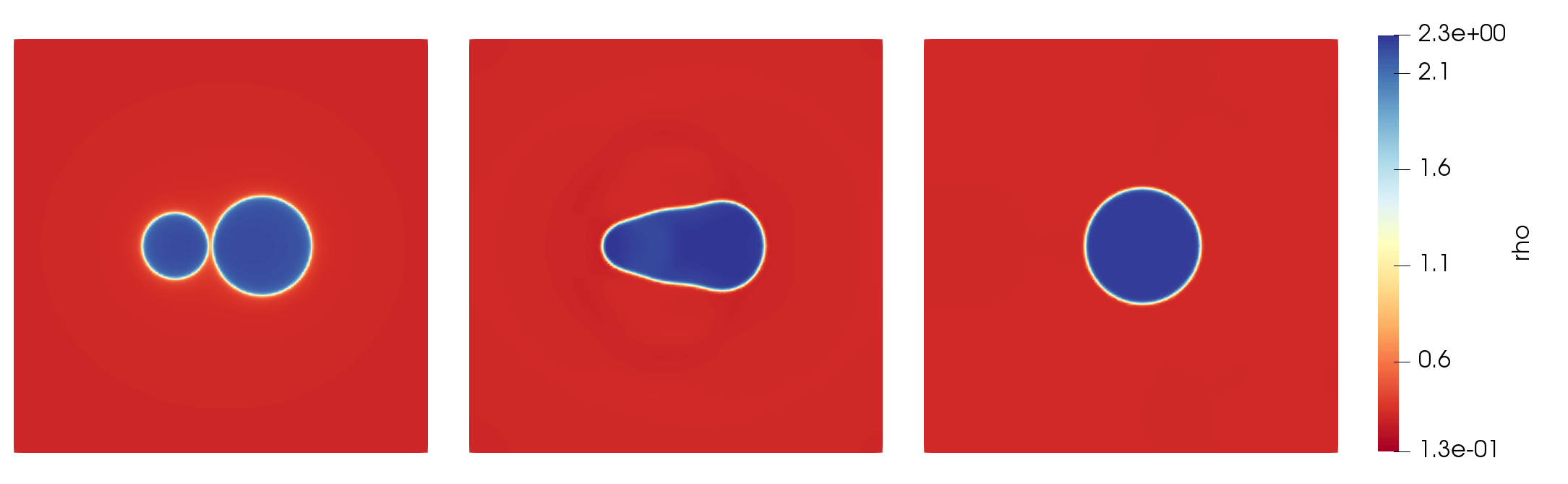}
    \caption{Density $\rho$ at times $t=0$, $t=0.2$, and $t=2$ for the merging of two droplets.}
    \label{fig:p2:merging}
\end{figure}
We can observe that the model handles topological changes easily.  However, the dynamics of the phase-field relaxation are determined by the mobility $\eta$. This is illustrated in \autoref{fig:p2:mergingenergy}, where the discrete energy $\hat E^{\Delta t}_h$  over time for different values of the mobility $\eta$ is plotted.\\
\begin{figure}[ht]
    \centering
   \includegraphics[height=.55\textwidth]{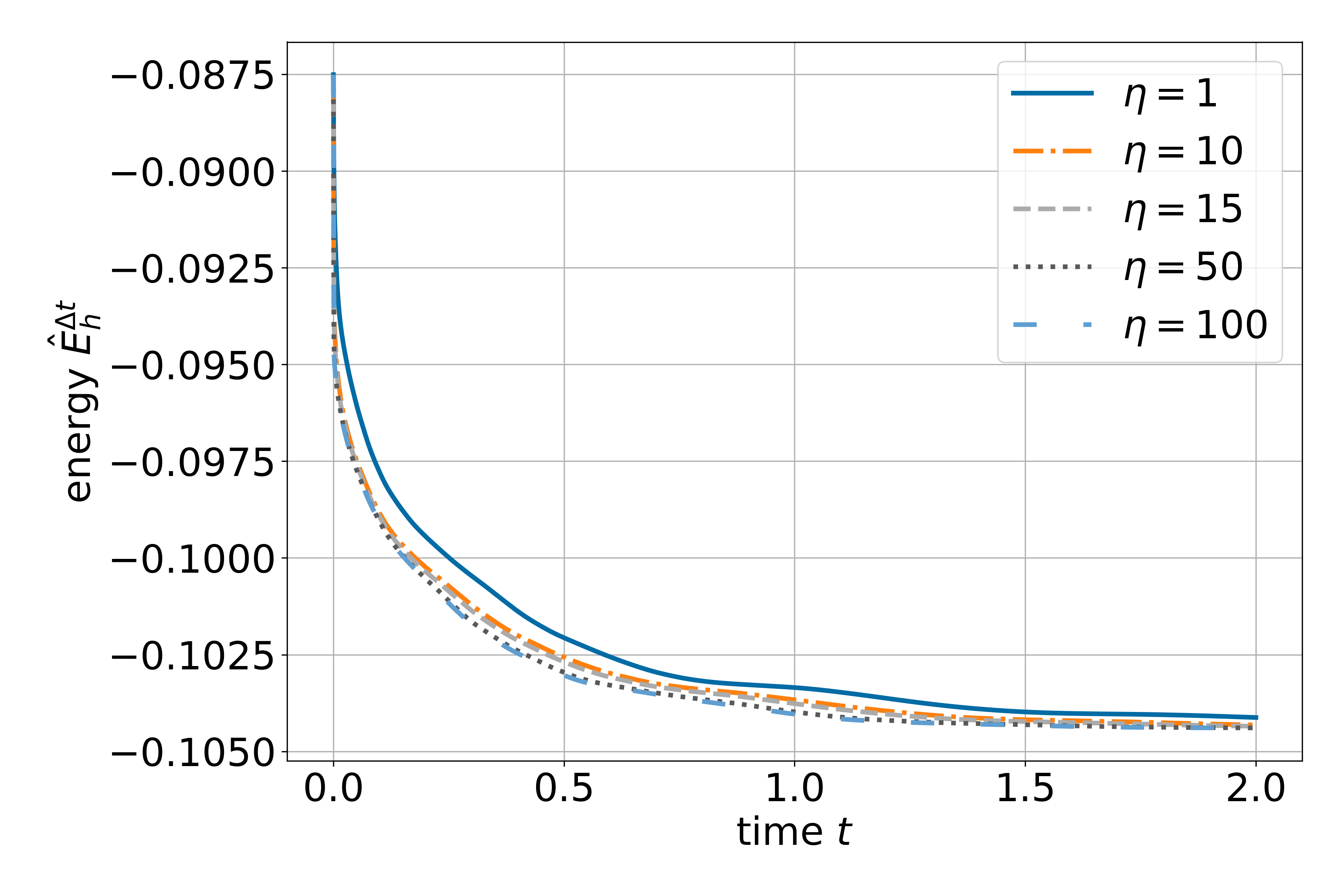}
    \caption{Time evolution of the discrete energy $  E= \hat E^{\Delta t}_h$ with different values for the mobility $\eta$. As expected, the energy decays faster with $\eta$ increasing.}
    \label{fig:p2:mergingenergy}
\end{figure}
We observe that the discrete energy   $\hat E^{\Delta t}_h$ from \eqref{energy:fulldiscrete} decreases, confirming  the statement of  Theorem \ref{thm:p2:energy_ineq_discrete} for our 
in space and time second-order scheme. Moreover, we see that the higher the value of $\eta$, the faster the energy dissipation.

\section{Conclusions}\label{chap:p2:conclusion}
We presented a  fully-discrete dG  method to solve the initial boundary value problem for the compressible NSAC system under isothermal conditions.  The method is shown to be of second-order in space and time, mass conservative, and to be unconditionally energy stable. 
For the energy stability we exploited the energetic
structure of the evolution equations: the  dG method relies on  an equivalent re-formulation of the original NSAC system such that  the 
nonlinear  variational derivatives of the free energy function can be expressed as elements of the dG ansatz space.\\
We conjecture that the same approach can be used to derive energy-stable schemes for related isothermal systems like Navier--Stokes--Cahn--Hilliard equations \cite{Mulet} or nonlocal 
diffuse-interface approaches \cite{Hitz}. More challenging is the extension of the approach to temperature-dependent two-phase systems or multiphase and multicomponent mixtures with complex viscous stress tensors and diffusion operators. The same applies for 
construction of energy-stable high-order methods in the regime of low Mach numbers which is typically given for at least the liquid phase.

\printbibliography

\end{document}